\documentclass{amsart}
\usepackage{amsmath, amssymb, amsthm}
\usepackage{url}

\newtheorem{thm}{Theorem}[section]
\newtheorem{pro}[thm]{Proposition}
\newtheorem{lem}[thm]{Lemma}
\newtheorem{cor}[thm]{Corollary}
\theoremstyle{definition}
\newtheorem{dfn}[thm]{Definition}
\newtheorem{exa}[thm]{Example}
\newtheorem{que}[thm]{Question}

\def\Aut{{\rm Aut}}
\def\C{{\mathcal{CDS}}}
\def\dbar{{\overline d}}
\def\Diff{{\rm Diff}}

\def\Fbar{{\overline F}}
\def\Ham{{\rm Ham}}
\def\Hameo{{\rm Hameo}}
\def\Hbar{{\overline H}}
\def\holine{{\overline h}}
\def\Homeo{{\rm Homeo}}
\def\id{{\rm id}}
\def\Int{{\rm Int}}
\def\Khat{{\widehat K}}
\def\mL{{\mathcal L}}
\def\mubar{{\overline{\mu}}}
\def\oneinfty{{(1,\infty)}}
\def\osc{{\rm osc}}
\def\PDiff{{\rm PDiff}}

\def\PHameo{{\rm PHameo}}
\def\Phihat{{\widehat \Phi}}
\def\phihat{{\widehat \phi}}
\def\PHomeo{{\rm PHomeo}}
\def\psihat{{\widehat \psi}}
\def\R{{\mathbb R}}
\def\Symp{{\rm Symp}}
\def\Sympeo{{\rm Sympeo}}
\def\TC{{\mathcal{TCDS}}}
\def\wH{{\widehat H}}
\def\Z{{\mathbb Z}}

\numberwithin{equation}{section}

\title[Topological contact dynamics I]{Topological contact dynamics I: symplectization and applications of the energy-capacity inequality}
\author[S.~M\"uller \& P.~Spaeth]{Stefan M\"uller and Peter Spaeth}
\address{University of Illinois at Urbana-Champaign, Urbana, IL 61801 \newline \indent Korea Institute for Advanced Study, Seoul 130-722, Republic of Korea}
\email{stefanm@illinois.edu}
\address{Penn State University, Altoona, PA 16601 \newline \indent Korea Institute for Advanced Study, Seoul 130-722, Republic of Korea}
\email{spaeth@psu.edu}


\begin{document}

\begin{abstract}
We introduce topological contact dynamics of a smooth manifold carrying a cooriented contact structure, generalizing previous work in the case of a symplectic structure \cite{mueller:ghh07} or a contact form \cite{banyaga:ugh11}.
A topological contact isotopy is not generated by a vector field; nevertheless, the group identities, the transformation law, and classical uniqueness results in the smooth case extend to topological contact isotopies and homeomorphisms, giving rise to an extension of smooth contact dynamics to topological dynamics.
Our approach is via symplectization of a contact manifold, and our main tools are an energy-capacity inequality we prove for contact diffeomorphisms, combined with techniques from measure theory on oriented manifolds.
We establish non-degeneracy of a Hofer-like bi-invariant pseudo-metric on the group of strictly contact diffeomorphisms constructed in \cite{banyaga:lci06}.
The topological automorphism group of the contact structure exhibits rigidity properties analogous to those of symplectic diffeomorphisms, including $C^0$-rigidity of contact and strictly contact diffeomorphisms.
\end{abstract}

\maketitle

\begin{center} {\em \small Dedicated to the memory of our friend Lee Jeong-eun.} \end{center}

\section{Introduction} \label{sec:intro}
Suppose a Hamiltonian diffeomorphism $\phi$ of a symplectic manifold $(W,\omega)$ is generated by a compactly supported Hamiltonian, and displaces a compact subset $K \subset \Int \, W$ containing an open ball.
The energy-capacity inequality from \cite{lalonde:gse95} implies
\begin{align} \label{eqn:energy-capacity-ineq}
	0 < \frac{1}{2} c (K) \leq E (\phi),
\end{align}
where the symplectic capacity $c (K)$ is the Gromov width of $K$, and $E (\phi)$ denotes the energy or Hofer norm of $\phi$.
The non-degeneracy of the Hofer metric \cite{hofer:tps90} follows immediately.
The displacement energy of $K$, or minimal energy required to displace $K$ from itself, is the infimum of $E (\phi)$ over all Hamiltonian diffeomorphisms $\phi$ as above, and by inequality~(\ref{eqn:energy-capacity-ineq}), it is bounded from below by one-half the capacity of $K$.
The existence of a symplectic capacity $c$ is sufficient to prove the Gromov--Eliashberg $C^0$-rigidity of symplectic diffeomorphisms, which means if a sequence of symplectic diffeomorphisms converges uniformly to another diffeomorphism of $W$, then the limit is again symplectic \cite{gromov:pdr86, gromov:shs87, eliashberg:tsw87}.
It is therefore consistent to define a symplectic homeomorphism (or topological automorphism) of the symplectic structure $\omega$ to be the limit of a $C^0$-convergent sequence of symplectic diffeomorphisms \cite{mueller:ghh07}.
This closure forms a subgroup of $\Homeo (W)$, denoted by $\Sympeo (W,\omega)$, and the Gromov--Eliashberg $C^0$-rigidity of symplectic diffeomorphisms can be stated succinctly $\Sympeo (W,\omega) \cap \Diff (W) = \Symp (W,\omega)$.

One goal of the present paper is to adapt these results to contact manifolds.
To that end, we prove an energy-capacity inequality for contact diffeomorphisms.

\begin{thm}[Contact energy-capacity inequality] \label{thm:energy-capacity-ineq}
Let $(M,\xi)$ be a contact manifold with a contact form $\alpha$.
Suppose the time-one map $\phi^1_H \in \Diff_0 (M,\xi)$ of a compactly supported smooth contact Hamiltonian $H \colon [0,1] \times M \to \R$ displaces a compact subset $K \subset \Int \, M$ containing an open ball.
Then there exists a constant $C > 0$, independent of the contact isotopy $\{ \phi_H^t \}$, its conformal factor $h \colon [0,1] \times M \to \R$ given by $(\phi_H^t)^* \alpha = e^{h (t,\cdot)} \alpha$, and the contact Hamiltonian $H$, such that
	\[ 0 < C e^{- | h |} \leq \| H \|_\alpha. \]
\end{thm}

The constant $C$ is determined by the displacement energy of the Cartesian product of the set $K$ with an interval in the symplectization $M \times \R$ of $(M,\alpha)$, and depends on the contact form $\alpha$.
See Sections~\ref{sec:contact-geometry}, \ref{sec:symplectization}, and \ref{sec:energy-capacity-ineq} for details.
As a consequence, we prove non-degeneracy of the bi-invariant pseudo-metric on the group of strictly contact diffeomorphisms defined by A.~Banyaga and P.~Donato in \cite{banyaga:lci06}.

\begin{thm} \label{thm:non-degeneracy}
Let $(M,\xi)$ be a contact manifold with a contact form $\alpha$.
The function
	\[ \Diff_0 (M,\alpha) \times \Diff_0 (M,\alpha) \to \R, \quad (\phi,\psi) \mapsto E (\phi^{-1} \circ \psi) \]
defines a bi-invariant metric on $\Diff_0 (M,\alpha)$.
\end{thm}

See Section~\ref{sec:bi-invariant-metric} for the definition of the contact energy $E$ and for the proof.
Moreover, we establish the following analog of symplectic $C^0$-rigidity for contact diffeomorphisms.

\begin{thm}[Contact $C^0$-rigidity] \label{thm:contact-rigidity}
Suppose $\phi_i$ is a sequence of contact diffeomorphisms of a contact manifold $(M,\xi)$, with $\phi_i^*\alpha = e^{h_i}\alpha$, where $\alpha$ is a contact form with $\ker \alpha = \xi$.
Further assume that the sequence $\phi_i$ converges uniformly on compact subsets to a homeomorphism $\phi$, and the sequence of functions $h_i$ converges to a continuous function $h$ uniformly on compact subsets.
If $\phi$ is smooth, then $h$ is smooth, and $\phi$ is a contact diffeomorphism with $\phi^*\alpha = e^h \alpha$.
\end{thm}

The uniform convergence of the conformal factors $h_i$ does not depend on the choice of contact form $\alpha$ with $\ker \alpha = \xi$.
We define the group $\Aut (M,\xi)$ of topological automorphisms of the contact structure $\xi$, analogous to the group $\Sympeo (W,\omega)$ above.
The contact $C^0$-rigidity theorem can then be stated in succinct terms $\Aut (M,\xi) \cap \Diff (M) = \Diff (M,\xi)$.
See Sections~\ref{sec:topo-contact-dynamics} and \ref{sec:automorphisms} for details.

This article is part of a series of papers on topological contact dynamics, and serves as an introduction to the theory.
We define topological contact isotopies and conformal factors, and show that both are determined uniquely by a topological contact Hamiltonian.
Composition and inversion of isotopies, as well as the transformation law, extend from smooth to topological contact dynamics.

In Section~\ref{sec:contact-geometry}, we review the necessary elements of contact geometry needed in subsequent sections, with focus on the dynamics of a contact vector field.
Similarly, Section~\ref{sec:ham-geometry} treats symplectic and Hamiltonian geometry, the even-dimensional analog to contact geometry, and the dynamics of a Hamiltonian vector field.
This section also contains a summary of compactly supported topological Hamiltonian dynamics, which we hope is more accessible than previous treatments of the subject.
Many outstanding monographs exist in the literature that develop smooth Hamiltonian and contact dynamics from a modern standpoint.
We have been influenced particularly by the books \cite{mcduff:ist98, hofer:sih94, polterovich:ggs01, banyaga:scd97, geiges:ict08, blair:rgc10}.
For the theory of topological Hamiltonian dynamics, we refer to the articles \cite{mueller:ghh07, mueller:ghl08, viterbo:ugh06, buhovsky:ugh11}.
The intimate relationship between contact and Hamiltonian dynamics via symplectization, explained in Section~\ref{sec:symplectization}, is the guiding principle in adapting topological Hamiltonian dynamics to topological contact dynamics.
The reader familiar with symplectic and contact geometry may skip these introductory sections at first reading, but should refer to them for our sign conventions and the notation used throughout this article.

The proof of the contact energy-capacity inequality in Theorem~\ref{thm:energy-capacity-ineq} requires a deep result in symplectic geometry.
After proving it in Section~\ref{sec:energy-capacity-ineq}, most of the more involved technical machinery moves to the backstage, allowing for short and elegant proofs of otherwise difficult results.
We use the symplectization of a contact manifold together with measure theory on orientable manifolds in a novel way.
A combination of these ingredients has several interesting consequences that are discussed in this paper and in its sequels.

Topological contact dynamics is then introduced in Section~\ref{sec:topo-contact-dynamics}.
This section contains the main results of topological contact dynamics in this paper, with some of the more involved proofs postponed to later sections.
We also explain topological Hamiltonian dynamics of the (non-compact) symplectization of a contact manifold.
An extensive motivation for the study of topological Hamiltonian dynamics can be found in \cite{mueller:ghh07, mueller:ghh08}, which applies almost verbatim in the contact case.
In addition to the applications in this article, the close relationship between the two theories via symplectization serves as another driving force for pursuing the study of topological contact dynamics.
Reducing dimension on the other hand, topological strictly contact dynamics of a regular contact manifold $M$ is closely related to topological Hamiltonian dynamics of the quotient of $M$ by the Reeb flow \cite{banyaga:ugh11}.
The sequel \cite{ms:tcd2} contains a detailed discussion.

In Section~\ref{sec:uniqueness-theorems}, we prove the previously stated main uniqueness theorems.
As in the Hamiltonian case in \cite{mueller:ghh07}, the energy-capacity inequality plays the key role in the proofs.
Detailed examples illustrating that all of the convergence hypotheses in the definition of topological contact dynamics and in the uniqueness theorems are necessary are given in Section~\ref{sec:examples}.
The group properties of topological contact dynamical systems, topological contact Hamiltonians, and topological contact isotopies and their time-one maps, are proved in Section~\ref{sec:topo-group}.
The proof of the transformation law can be found there as well.

In Section~\ref{sec:bi-invariant-metric}, we prove the existence of a bi-invariant metric on the group of strictly contact diffeomorphisms, with no restrictions on the contact form.
This generalizes a theorem of Banyaga and Donato to any contact form $\alpha$.
However, for the group of contact isotopies, we show by example the failure of the triangle inequality, and thus the distance on the group of strictly contact isotopies does not extend in this case.
Our construction is local in nature, and thus applies to any contact manifold.
Section~\ref{sec:automorphisms} studies the groups of topological automorphisms of a contact structure $\xi$ and of a contact form $\alpha$, which are the analogs to the group $\Sympeo (W,\omega)$ of topological automorphisms of a symplectic structure $\omega$.
In this section, we also prove the $C^0$-rigidity of contact and strictly contact diffeomorphisms, and other consequences of the properties of topological automorphisms.
See also Section~\ref{sec:smooth-conformal-factors}.
A brief outlook into the sequels to this work is undertaken in the final Section~\ref{sec:sequels}.

Some of the sections can be read mostly independently of the rest of the paper.
We mention in particular Sections~\ref{sec:energy-capacity-ineq}, \ref{sec:uniqueness-theorems}, \ref{sec:examples}, \ref{sec:bi-invariant-metric}, \ref{sec:automorphisms}, and \ref{sec:smooth-conformal-factors}, which are of particular relevance in smooth contact dynamics.
This first part in our series of papers on topological contact dynamics lays the foundations for most later applications.
Its guiding principles are symplectization and consequences of the energy-capacity inequality.
Other results and applications are organized under different umbrellas and postponed to one of the two sequels \cite{ms:tcd2, ms:tcd3}.

\section{Review of contact geometry and contact dynamics} \label{sec:contact-geometry}
Let $(M,\xi)$ be a smooth manifold of dimension $2 n - 1$ equipped with a cooriented nowhere integrable field of hyperplanes $\xi \subset TM$.
The \emph{contact structure} $\xi$ can be written as $\xi = \ker \alpha$, where the \emph{contact form} $\alpha$ is a smooth one-form on $M$ such that $\nu_\alpha = {\alpha \wedge (d\alpha)^{n - 1}} \neq 0$.
Unless mentioned otherwise, the manifold $M$ is always assumed to be closed, i.e.\ compact and without boundary.
See the remarks at the end of Section~\ref{sec:topo-contact-dynamics} for the case of open manifolds, that is, those manifolds that are not closed.
For simplicity, we assume throughout that $M$ is connected.
We fix a coorientation of $\xi$, and hence an orientation of $M$.
Then any other contact form $\alpha'$ on $(M,\xi)$ can be written $\alpha' = e^f \alpha$ for a smooth function $f$ on $M$.
A diffeomorphism $\phi$ is called \emph{contact} if it preserves the contact structure, and this is equivalent to the existence of a smooth function $h \colon M \to \R$ such that
\begin{align} \label{eqn:conformal-factor}
	\phi^* \alpha = e^h \alpha.
\end{align}
We denote the group of contact diffeomorphisms by $\Diff (M,\xi)$, and the subgroup of contact diffeomorphisms isotopic to the identity inside $\Diff (M,\xi)$ by $\Diff_0 (M,\xi)$.
In their book \cite{mcduff:ist98}, D.~McDuff and D.~Salamon ask if a $C^0$-characterization of contact diffeomorphisms exists.
Non-squeezing results and the existence of capacities depend in a more subtle way on the topology of the underlying contact manifold.
We refer the reader to \cite{eliashberg:gct06}.

The \emph{Reeb vector field} $R$ defined by $\alpha$ is the unique vector field on $M$ in the kernel of $d\alpha$ satisfying $\iota (R) \alpha = 1$, where $\iota$ denotes interior multiplication or contraction of a differential form by a smooth vector field.
An isotopy $\Phi = \{ \phi_t \}_{0 \le t \le 1}$ is a \emph{contact isotopy} if there exists a smooth family of functions $h_t \colon M \to \R$ satisfying
\begin{align} \label{eqn:contact-isotopy}
	\phi_t^*\alpha = e^{h_t} \alpha.
\end{align}
$\Phi$ is contact if and only if the smooth vector fields $X_t = (\frac{d}{dt} \phi_t) \circ \phi_t^{-1}$ form a family of \emph{contact vector fields}, meaning the Lie derivative of $\alpha$ along $X_t$ satisfies $\mL_{X_t} \alpha = \mu_{X_t} \alpha$, for a smooth family of functions $\mu_{X_t} \colon M \to \R$.
In contrast to a symplectic isotopy, a contact isotopy is always `Hamiltonian', and the contact Hamiltonian function $H \colon [0,1] \times M \to \R$ is determined at each time $t$ by the equation $\iota(X_t) \alpha = H_t = H (t,\cdot)$.
Conversely, given a smooth family of functions $H_t \colon M \to \R$, the equations
\begin{align} \label{eqn:contact-ham}
	\iota(X_t) \alpha = H_t \quad \text{and} \quad \iota(X_t) d\alpha = (R . H_t) \alpha - dH_t
\end{align}
define a smooth family of contact vector fields $X_t$, whose flow satisfies equation~(\ref{eqn:contact-isotopy}), and $\mu_{X_t} = R . H_t$.
Here we write $R . H = dH (R)$ for the derivative of the smooth function $H$ in the direction of the Reeb vector field $R$.
The function $h$ satisfying equation~(\ref{eqn:conformal-factor}) is called the \emph{conformal factor} of the contact diffeomorphism $\phi$, and the time-dependent function $h \colon [0,1] \times M \to \R$ defined by equation~(\ref{eqn:contact-isotopy}) and $h (t,\cdot) = h_t$ is called the conformal factor of the isotopy $\Phi$.
It is related to $H$ through the identity
\begin{align} \label{eqn:formula-conformal-factor}
	h_t = \int_0^t (R . H_s) \circ \phi^s_H \, ds.
\end{align}
A contact isotopy $\Phi$ will often be denoted by $\Phi_H  = \{ \phi_H^t \}$ provided equation~(\ref{eqn:contact-ham}) holds, and similarly for the contact vector field $X_H$.
The group of smooth contact isotopies is labeled $\PDiff(M,\xi)$.
The notation $H \mapsto \Phi_H$ and $H \mapsto \phi^1_H$ is short-hand for writing $H$ generates the isotopy $\Phi_H$ and the time-one map $\phi^1_H$, respectively.

A contact diffeomorphism $\phi$ is called \emph{strictly contact} if it preserves not only the contact structure $\xi$ but also the contact form $\alpha$, that is, $\phi^* \alpha = \alpha$.
If $\phi$ is a contact diffeomorphism satisfying equation~(\ref{eqn:conformal-factor}), then $\phi^* \nu_\alpha = e^{n h} \nu_\alpha$.
Thus a contact diffeomorphism is strictly contact if and only if it preserves the volume form $\nu_\alpha$.
A contact isotopy $\{ \phi_t \}$ is strictly contact if $\phi_t$ is strictly contact for each time $t$, or equivalently, its conformal factor $h \colon [0,1] \times M \to \R$ is identically zero.
A smooth function $H \colon [0,1] \times M \to \R$ is called \emph{basic} if $R . H_t = 0$, or $H_t$ is invariant under the Reeb flow, for all $t$.
Then a contact isotopy $\Phi_H$ is strictly contact if and only if its generating contact Hamiltonian $H$ is basic.
The groups of strictly contact diffeomorphisms and strictly contact isotopies are written $\Diff (M,\alpha) \subset \Diff (M,\xi)$ and $\PDiff (M,\alpha) \subset \PDiff (M,\xi)$, respectively.

We recall some other facts about diffeomorphisms, vector fields, and contact isotopies \cite{libermann:sga87}.
The proofs of the next two lemmas are widely known.

\begin{lem} \label{lem:conformal-calculus}
Let $\phi$ and $\psi$ be diffeomorphisms of a smooth manifold $M$, $\beta$ a differential form, and $X$ a vector field on $M$.
Then $\iota (\phi_* X ) \beta = (\phi^{-1})^* (\iota (X) \phi^* \beta)$.
If there exist smooth functions $h$ and $g$ on $M$, such that $\phi^* \beta = e^h \beta$ and $\psi^* \beta = e^g \beta$, then $(\phi \circ \psi)^* \beta = e^{h \circ \psi + g} \beta$, and $(\phi^{-1})^* \beta = e^{-h \circ \phi^{-1}} \beta$.
\end{lem}

\begin{lem}[Contact Hamiltonian group structure] \label{lem:ham-gp-str}
Suppose $H \mapsto \Phi_H$ and $F \mapsto \Phi_F$.
Then the following smooth functions generate the indicated contact isotopies.
\begin{align*}
	& H \# F \mapsto \Phi_H \circ \Phi_F,  & (H \# F)_t = H_t + \left( e^{h_t} \cdot F_t \right) \circ (\phi_H^t)^{-1}, \\
	& \Hbar \mapsto \Phi_H^{-1},  & \Hbar_t = - e^{- h_t} \cdot \left( H_t \circ \phi_H^t \right), \\
	& \Hbar \# F \mapsto \Phi_H^{-1} \circ \Phi_F,  & (\Hbar \# F)_t = e^{- h_t} \cdot \left( (F_t - H_t) \circ \phi_H^t \right), \\
	& K \mapsto \phi^{-1} \circ \Phi_H \circ \phi, & K_t = e^{- g} \left( H_t \circ \phi \right),
\end{align*}
for $\phi \in \Diff (M,\xi)$ with $\phi^*\alpha = e^g \alpha$.
Here composition $\circ$ and inversion is to be understood as composition and inversion of the diffeomorphisms at each time $t$.
\end{lem}

We call a triple $(\Phi, H, h)$ a smooth \emph{contact dynamical system} if $\Phi = \Phi_H$ is a smooth contact isotopy with contact Hamiltonian $H$ and conformal factor $h$, and denote the group of such triples by $\C (M,\alpha)$.
The subgroup of smooth \emph{strictly contact dynamical systems} $(\Phi, H, 0)$ is denoted by $\mathcal{SCDS} (M,\alpha)$.
As we have seen above, the smooth isotopy $\Phi$ and the contact form $\alpha$ together uniquely determine the contact Hamiltonian $H$ and the conformal factor $h$, and conversely, given a contact form $\alpha$, the Hamiltonian $H$ uniquely determines both $\Phi$ and $h$.
This correspondence depends on the choice of contact form.
However, the groups of contact diffeomorphisms and of smooth contact isotopies do not.
More precisely, we have the following lemma, whose proof is straightforward.

\begin{lem} \label{lem:change-contact-form}
Let $\alpha$ and $\alpha' = e^f \alpha$ be two contact forms on a contact manifold $(M,\xi)$.
If $(\Phi, H, h) \in \C (M,\alpha)$ is a smooth contact dynamical system with respect to the contact form $\alpha$, then $(\Phi, e^f H, h + (f \circ \Phi - f)) \in \C (M,\alpha')$ is a smooth contact dynamical system with respect to the contact form $\alpha'$.
\end{lem}

Here the notation $f \circ \Phi$ stands for the function whose value at $(t,x) \in [0,1] \times M$ is $(f \circ \Phi) (t,x) = f (\phi_t (x))$.
When a contact form $\alpha$ on $(M,\xi)$ is chosen, we always assume the contact Hamiltonian and conformal factor of a smooth contact isotopy are determined by this contact form $\alpha$.
Moreover, even when no contact form is selected explicitly, writing $\Phi = \Phi_H$ for a contact isotopy $\Phi$, a contact dynamical system $(\Phi, H, h)$, or calling $H$ the contact Hamiltonian of $\Phi$ and $h$ its conformal factor, implies the choice of a contact form $\alpha$, which is fixed for the remainder of a particular statement or short discussion, unless explicit mention is made to the contrary.

The \emph{length} \cite{banyaga:ugh11} of a contact isotopy $\Phi = \Phi_H$ of $(M, \xi)$ is defined via its contact Hamiltonian $H \colon [0,1] \times M \to \R$ to be
\begin{align} \label{eqn:contact-length}
	\ell_\alpha (\Phi) = \| H \|_\alpha = \int_0^1 \left( \max_{x \in M} H (t,x) - \min_{x \in M} H (t,x) + \left| c_\alpha (H_t) \right| \right) dt,
\end{align}
where $c_\alpha$ denotes the average value of a function $M \to \R$ with respect to the measure induced by the volume form $\nu_\alpha$, i.e.
\begin{align} \label{eqn:average-value}
	c_\alpha (H_t) = \frac{1}{\int_M \nu_\alpha} \cdot \int_M H_t \, \nu_\alpha.
\end{align}
We also refer to $\| H \|_\alpha$ as the norm of $H$.
Note that $\ell_\alpha$ depends on the choice of contact form $\alpha$ since this choice determines the contact Hamiltonian $H$ of a contact isotopy $\Phi$, and $\| \cdot \|_\alpha$ also depends on $\alpha$ via the definition of the average value $c_\alpha$ with respect to the volume form $\nu_\alpha$.
However, the norms and length functions resulting from these choices are all mutually equivalent.
To see this, we make a simple observation.
For an autonomous function $H$ on $M$, the oscillation $\max H(x) - \min H (x)$ on $M$ is denoted by $\osc (H)$.

\begin{lem} \label{lem:max-norm}
The norm $\osc (H) + | c_\alpha (H) |$ and the maximum norm $| H | = \max | H (x) |$ of functions $M \to \R$ are equivalent.
\end{lem}

\begin{proof}
$| H | \le \osc (H) + | c_\alpha (H) | < 3 | H |$.
These inequalities are sharp.
\end{proof}

We prefer the norm $\| \cdot \|_\alpha$ defined by equation~(\ref{eqn:contact-length}), since it closely resembles the choice of norm of a Hamiltonian function on a symplectic manifold in the next section.
This relation is most prominent when $(M,\alpha)$ is the total space of a principle $S^1$-bundle over an integral symplectic manifold \cite{banyaga:ugh11}.
The choice of norm in the Hamiltonian case is explained in \cite{mueller:ghh07}.

\begin{lem} \label{lem:equiv-norms}
Let $\alpha$ be a contact form on $(M,\xi)$, and $f$ be a smooth function on $M$.
Then there exist positive constants $c (f)$ and $C (f)$ that depend only on $f$, such that
	\[ c (f) \cdot \| H \|_\alpha \le \| e^f H \|_{\alpha'} \le C (f) \cdot \| H \|_\alpha \]
for any function $H \colon [0,1] \times M \to \R$, and any contact form $\alpha' = e^g \alpha$.
In particular, the norms $\| \cdot \|_\alpha$ and $\| \cdot \|_{\alpha'}$ as well as the induced length functions $\ell_\alpha$ and $\ell_{\alpha'}$ on the group of smooth contact isotopies are equivalent.
Moreover, if a collection of smooth functions $f_i$ is uniformly bounded independently of $i$, $| f_i | < c < \infty$, then the constants $c (f_i)$ and $C (f_i)$ can be chosen independently of $i$.
\end{lem}

\begin{proof}
By Lemma~\ref{lem:max-norm}, one can choose $c (f) = \frac{1}{3} e^{- | f |}$ and $C (f) = 3 e^{| f |}$.
The choices $f = 0$ and $f = g$ prove the equivalence of the norms and length functions, respectively.
\end{proof}

For reasons that will soon become apparent, for conformal factors $h \colon [0,1] \times M \to \R$ however, we instead work with the maximum norm
\begin{align} \label{eqn:max-norm}
	| h | = \max_{0 \le t \le 1} \max_{x \in M} \left| h (t,x) \right|.
\end{align}
Given a smooth contact isotopy $\Phi$, its conformal factor $h$ also depends on the contact form $\alpha$, and transforms under a change of contact form according to Lemma~\ref{lem:change-contact-form}.
The behavior of this change of the conformal factors of a convergent sequence is explained in Lemma~\ref{lem:alpha-indep} below.

A choice of Riemannian metric $g_M$ on $M$ gives rise to a distance function $d_M$ on $M$, and thus on the spaces $\Homeo(M)$ of homeomorphisms and $\PHomeo (M)$ of isotopies of homeomorphisms: for two homeomorphisms $\phi$ and $\psi$ of $M$, we have
	\[ d_M (\phi, \psi) = \max_{x \in M} d_M (\phi (x), \psi (x) ), \]
and this \emph{uniform distance} induces the compact-open topology.
In particular, the topology is independent of the initial choice of Riemannian metric.
The metric $d_M$ is not complete, but it gives rise to a complete \emph{$C^0$-metric} $\dbar_M$ that induces the same topology, where
	\[ \dbar_M (\phi, \psi) = d_M (\phi, \psi) + d_M (\phi^{-1}, \psi^{-1}). \]
In fact, the Finsler norms induced by different choices of Riemannian metrics $g_1$ and $g_2$ are locally equivalent (in a coordinate chart), so by compactness of $M$, they gives rise to equivalent distance functions $d_1$ and $d_2$ on $M$.
The induced metrics $d_1$ and $d_2$ on $\Homeo(M)$ and $\PHomeo (M)$ associated to different choices of Riemannian metrics as well as the corresponding metrics $\dbar_1$ and $\dbar_2$ are therefore also equivalent.

Both metrics $d_M$ and $\dbar_M$ define distances between isotopies $\Phi = \{ \phi_t \}$ and $\Psi = \{ \psi_t \}$, equal to the maximum over all times $t$ of the distances of the time-$t$ maps $\phi_t$ and $\psi_t$, and again the metric
	\[ \dbar_M (\Phi, \Psi) = \max_{0 \le t \le 1} \dbar_M (\phi_t, \psi_t) \]
is complete, while $d_M (\Phi, \Psi)$ is not.
However, if a sequence $\Phi_i$ of isotopies converges uniformly to an isotopy of homeomorphisms $\Phi$, or in other words, a limit in $\PHomeo (M)$ with respect to the distance $d_M$ exists, then this sequence is also Cauchy with respect to the distance $\dbar_M$, and $C^0$-converges to the isotopy $\Phi$.
Moreover, composition and inversion are continuous with respect to the $C^0$-metric.
The same remarks apply to sequences of homeomorphisms.

A Cauchy sequence of smooth functions with respect to the maximum norm $| \cdot |$ in equation~(\ref{eqn:max-norm}) converges to a continuous time-dependent function $h \colon [0,1] \times M \to \R$.
On the other hand, a Cauchy sequence of smooth contact Hamiltonians converges with respect to the norm $\| \cdot \|_\alpha$ in equation~(\ref{eqn:contact-length}) to a so called $L^\oneinfty$-function $H \colon [0,1] \times M \to \R$.
This function may not be continuous but only $L^1$ in the time variable $t \in [0,1]$.
However, by standard arguments from measure theory, $H_t$ is defined for almost all $t \in [0,1]$, and is a continuous function of the space variable $x \in M$ for each such $t$.
Thus it can be thought of as an element of the space of functions $L^1 ([0,1],C^0 (M))$ of $L^1$-functions of the unit interval taking values in the space $C^0 (M)$ of continuous functions of $M$.
Strictly speaking, such a function should be thought of as an equivalence class of functions, where two functions are considered equivalent if and only if they agree for almost all $t \in [0, 1]$, but as is customary in measure theory, we will mostly disregard this subtlety in our treatment, and speak of an $L^\oneinfty$-function when it can not lead to any confusion.

\section{Review of Hamiltonian geometry and Hamiltonian dynamics} \label{sec:ham-geometry}
Let $(W,\omega)$ be a smooth manifold of even dimension $2 n$ equipped with a \emph{symplectic form} $\omega$.
That is, the two-form $\omega$ is closed and non-degenerate, i.e.\ $\omega^n \not= 0$, and in particular induces an orientation of $W$.
Again unless explicit mention is made to the contrary, $W$ is assumed to be closed, and for simplicity, we only consider manifolds that are connected.
If $\phi^* \omega = e^h \omega$ for a smooth function $h$ on $W$, that is, if a diffeomorphism conformally rescales the symplectic form $\omega$, and $n > 1$, then $h$ must be constant because $\omega$ is closed, and by compactness, this constant is equal to zero.
A diffeomorphism $\phi$ that preserves the \emph{symplectic structure} $\omega$ is called a \emph{symplectic diffeomorphism}.
The group of symplectic diffeomorphisms of $(W,\omega)$ is denoted $\Symp (W,\omega) = \{ \phi \in \Diff (W) \mid \phi^* \omega = \omega \}$, and its identity component is $\Symp_0 (W,\omega)$.

The unique feature of a smooth \emph{Hamiltonian} dynamical system of $(W,\omega)$ is that it is defined up to normalization by a smooth time-dependent function $H \colon [0,1] \times W \to \R$.
We recall the aspects of smooth Hamiltonian dynamics that we will need in the following, and of topological Hamiltonian dynamics to put the approach taken in Section~\ref{sec:topo-contact-dynamics} in perspective.

The non-degeneracy of the symplectic form $\omega$ implies that to a smooth time-dependent Hamiltonian function $H \colon [0,1]\times W \to \R$ is associated a unique time-dependent vector field $X_H = \{ X_H^t \}$, defined by the equation
	\[ \iota (X_H^t) \omega = d H_t, \]
and the isotopy $\Phi_H = \{ \phi_H^t \}$ generated by $X_H^t$ is by definition the solution to the ordinary differential equation
	\[ \frac{d}{dt} \phi_H^t = X_H^t \circ \phi_H^t, \quad \phi_H^0 = \id. \]
On the other hand, suppose for a smooth family of vector fields $X_t$ generating an isotopy $\Phi = \{ \phi_t \}$, the one-form $\iota (X_t) \omega$ is exact at each time $t$.
Then there exists a unique normalized smooth Hamiltonian function $H \colon [0,1] \times W \to \R$ satisfying the identity $\iota (X_t) \omega = dH_t$.
A time-dependent function is \emph{normalized} if it has average value zero with respect to the Liouville measure induced by the volume form $\omega^n$ at each time $t$.
(If $M$ is open, the Hamiltonians under consideration are compactly supported in the interior.)
We assume throughout that all Hamiltonians are normalized, so that there is a one-to-one correspondence between smooth Hamiltonian isotopies and smooth Hamiltonian functions.
As in the contact case in the previous section, we write $H \mapsto \Phi$ when the isotopy $\Phi = \Phi_H$ is generated by the smooth Hamiltonian $H$, and similarly $H \mapsto \phi$ if the time-one map of the isotopy is $\phi = \phi_H^1$.

Lemma~\ref{lem:ham-gp-str} holds verbatim, except that the conformal factors are all identically zero in this case.
We denote by $\mathcal{HDS} (W,\omega)$ the collection of pairs $(\Phi, H)$, where $\Phi$ is a smooth Hamiltonian isotopy, generated by the smooth normalized Hamiltonian function $H \colon [0,1] \times W \to \R$, and call $(\Phi, H)$ a \emph{smooth Hamiltonian dynamical system} of the symplectic manifold $(W,\omega)$.
The spaces of smooth Hamiltonian dynamical systems, smooth Hamiltonian functions, and smooth Hamiltonian isotopies and their time-one maps, all form groups.
The latter is denoted by $\Ham (W,\omega)$, and is a normal subgroup of the group of symplectic diffeomorphisms of $(W,\omega)$.

The norm used to define the metric on the space of normalized Hamiltonian functions is the usual \emph{Hofer length} or \emph{Hofer norm}
\begin{align} \label{eqn:hofer-length}
	\ell_{\rm Hofer} (\Phi_H) = \| H \|_{\rm Hofer} = \int_0^1 \left( \max_{x \in W} H (t,x) - \min_{x \in W} H (t,x) \right) dt,
\end{align}
which is in fact the same as equation~(\ref{eqn:contact-length}) since the Hamiltonian $H$ has mean value zero.

The vital aspects of the smooth theory are the one-to-one correspondence between isotopies and Hamiltonian functions, the group identities, and the energy-capacity inequality.
The latter two play a crucial role in establishing the Hofer norm on $\Ham (W,\omega)$, where by definition
	\[ \| \phi \|_{\rm Hofer} = E (\phi) = \inf_{H \mapsto \phi} \| H \|_{\rm Hofer}. \]

The first step towards defining a {\em topological} Hamiltonian dynamical system is the proper choice of metric on the group of smooth Hamiltonian isotopies.
Indeed, only a combination of the dynamical Hofer length and the topological $C^0$-distance yields such a metric $d_{\rm ham}$ on the group of smooth Hamiltonian isotopies.
Let $\Phi_H$ and $\Phi_F$ be two smooth Hamiltonian isotopies, and set
	\[ d_{\rm ham} (\Phi_H, \Phi_F) = \dbar_W (\Phi_H, \Phi_F) + \| \Hbar \# F \| = \dbar_W (\Phi_H, \Phi_F) + \| H - F \|. \]
This \emph{Hamiltonian metric} is no longer bi-invariant, but the upshot is that the completion with respect to $d_{\rm ham}$ of the group $\mathcal{HDS}(W,\omega)$ results in a group of pairs $(\Phi, H)$, where $\Phi$ is an isotopy of homeomorphisms of $M$, and $H$ is an $L^\oneinfty$-function.
See section~3.2 in \cite{mueller:ghh08} for a detailed explanation of the choice of completion.
An isotopy $\Phi = \{ \phi_t \}$ of homeomorphisms is a \emph{topological Hamiltonian isotopy} of $(W,\omega)$, if there exists a $d_{\rm ham}$-Cauchy sequence of smooth Hamiltonian isotopies $\Phi_{H_i}$ that uniformly converges to $\Phi$.
The $L^\oneinfty$-limit $H$ of the sequence of normalized smooth Hamiltonians $H_i$ is called a \emph{topological Hamiltonian function}, and a homeomorphism $\phi$ is a \emph{Hamiltonian homeomorphism} if it is the time-one map of a topological Hamiltonian isotopy \cite{mueller:ghh07}.
We denote the collection $(\Phi, H)$ of such \emph{topological Hamiltonian dynamical systems} by $\mathcal{THDS}(W,\omega)$.
By \cite{mueller:ghh07}, $\mathcal{THDS}(W,\omega)$ has the structure of a topological group, whose group operations project continuously to the spaces of topological Hamiltonian isotopies, functions, and homeomorphisms, and the isotopy associated to a topological Hamiltonian function is unique.
The converse that the topological Hamiltonian function associated to a topological Hamiltonian isotopy is unique, is proven in increasing generality in \cite{viterbo:ugh06} and \cite{buhovsky:ugh11}.
The groups of topological Hamiltonian isotopies and Hamiltonian homeomorphisms are denoted by $\PHameo(W,\omega)$ and $\Hameo(W,\omega)$, respectively.
The notion of topological Hamiltonian isotopy and Hamiltonian homeomorphism has been generalized to symplectic isotopies and their time-one maps by Banyaga in the article \cite{banyaga:gss10}.

\section{Symplectization} \label{sec:symplectization}
Let $\alpha$ be a contact form on $M$ defining the contact structure $\xi$, i.e.\ $\ker \alpha = \xi$.
The \emph{symplectization} $(W,\omega)$ of $(M,\alpha)$ is the exact symplectic manifold
	\[ \left( M \times \R, - d ( e^\theta \pi_1^* \alpha ) \right), \]
where $\theta$ is the coordinate on $\R$, and $\pi_1 \colon M \times \R \to M$ is the projection to the first factor.
The exact symplectic diffeomorphism class of $(W,\omega)$ depends only on the contact structure $\xi$ and not on the choice of contact form $\alpha$.
Indeed, if $\alpha' = e^f \alpha$ is any other contact form on $(M,\xi)$, and $(W,\omega')$ denotes the symplectization of $(M,\alpha')$, then the diffeomorphism $\phi_f \colon (W,\omega') \to (W,\omega)$ given by mapping $(x,\theta)$ to $(x,\theta + f (x))$ is exact symplectic, i.e.\ $\phi_f^* (e^\theta \pi_1^* \alpha) = e^\theta \pi_1^* \alpha'$.

A contact diffeomorphism $\phi$ lifts to a symplectic diffeomorphism
\begin{align} \label{eqn:lifted-diffeo}
	\phihat (x,\theta) = (\phi (x), \theta - h (x))
\end{align}
of $(W,\omega)$, where $h$ is the conformal factor of $\phi$ given by $\phi^* \alpha = e^h \alpha$.
Conversely, a diffeomorphism $\phihat$ of $W$ of the form in equation~(\ref{eqn:lifted-diffeo}) is symplectic if and only if $\phi$ is a contact diffeomorphism of $M$ with $\phi^* \alpha = e^h \alpha$.
A contact isotopy $\Phi_H = \{ \phi_t \}$ then lifts to a Hamiltonian isotopy $\Phi_\wH = \{ \phihat_t \}$ of $(W,\omega)$, generated by the Hamiltonian
\begin{align} \label{eqn:admissible-ham}
	\wH (t, x, \theta) = e^\theta H (t, x).
\end{align}
A diffeomorphism (or homeomorphism) of the form as in equation~(\ref{eqn:lifted-diffeo}), where $\phi$ is a diffeomorphism (homeomorphism) of $M$, and $h$ a smooth (continuous) function on $M$, is called \emph{admissible}.
An isotopy $\{ \phi_t \}$ is called admissible if $\phi_t$ is admissible for each $t$, and if it is Hamiltonian, then its Hamiltonian of the form as in equation~(\ref{eqn:admissible-ham}) is also called admissible.

Let $g_M$ be a Riemannian metric on $M$, and recall from Section~\ref{sec:contact-geometry} the corresponding distances $d_M$ and $\dbar_M$ on $\Homeo (M)$ and $\PHomeo (M)$, which both induce the compact-open topology.
The Riemannian metric $g_M$ lifts to the split Riemannian metric $g_W = \pi_1^*g_M + d\theta \otimes d\theta$ on the symplectization $W$.
Given two admissible homeomorphisms $\phihat (x,\theta) = (\phi (x), \theta - h (x))$ and $\psihat (x,\theta) = (\psi (x), \theta - g (x))$ of $W$, the sums
	\[ d_W (\phihat, \psihat) = d_M (\phi, \psi) + | h - g | \]
and $\dbar_W (\phihat, \psihat) = \dbar_M (\phi, \psi) + | h - g | + | h \circ \phi^{-1} - g \circ \psi^{-1} |$
are finite, and the two distances $d_W$ and $\dbar_W$ are metrics on the group of admissible homeomorphisms of $W$, with the latter being complete.
In particular, given a sequence $\phi_i$ of contact diffeomorphism with $\phi_i^* \alpha = e^{h_i} \alpha$, the sequence of lifts $\phihat_i$ defined by equation~(\ref{eqn:lifted-diffeo}) is $\dbar_W$-Cauchy, if and only if the sequence $\phi_i$ is $\dbar_M$-Cauchy and the sequence of functions $h_i$ is uniformly Cauchy.
Moreover, the sequence $\phi_i$ then $C^0$-converges to a homeomorphism $\phi$, the smooth functions $h_i$ converge uniformly to a continuous function $h$, and the sequence $\phihat_i$ $C^0$-converges to the homeomorphism $\phihat$ of $W$ given by $\phihat (x,\theta) = (\phi (x), \theta - h (x))$.
Similarly one obtains a metric $d_W$ and a complete metric $\dbar_W$ on the group of admissible continuous isotopies of homeomorphisms, and for a sequence of contact isotopies of $M$, the lifted Hamiltonian isotopies are $C^0$-Cauchy, if and only if the contact isotopies are $C^0$-Cauchy and their conformal factors are uniformly Cauchy.
The limits are again related by the identity~(\ref{eqn:lifted-diffeo}) at each time $t$.

In Section~\ref{sec:ham-geometry}, we discussed Hamiltonian isotopies of \emph{non-compact} manifolds $W$, generated by \emph{compactly supported} Hamiltonians.
Clearly an admissible Hamiltonian is never compactly supported, and its oscillation is not finite, unless it vanishes identically.
Similarly, it does not make sense to normalize an admissible Hamiltonian by means of its average value on $W$.
However, an admissible Hamiltonian isotopy has a unique admissible Hamiltonian, and it is possible to define a norm suitable for admissible Hamiltonians.
For $a < b$ real numbers, let $K_{a,b} = M \times [a,b]$ be a compact subset of $W$, and restrict to it the oscillation of functions on $[0,1] \times W$,
\begin{align*}
	\| \wH \|_\alpha^{a,b} & = \int_0^1 \left( \max_{a \le \theta \le b} \max_{x \in M} \wH (t,x,\theta) - \min_{a \le \theta \le b} \min_{x \in M} \wH (t,x,\theta) \right) dt \\
	& = \int_0^1 \left( \max_{a \le \theta \le b} \left( e^\theta \max_{x \in M} H (t,x) \right) - \min_{a \le \theta \le b} \left( e^\theta \min_{x \in M} H (t,x) \right) \right) dt.
\end{align*}
When restricted to admissible Hamiltonians, the functions $\| \cdot \|_\alpha^{a,b}$ define norms, and different choices of $a$ and $b$ give rise to equivalent norms.
Moreover,
	\[ \min(e^b - e^a, e^a) \cdot \| H \|_\alpha \le \| \wH \|_\alpha^{a,b} \le e^b \cdot \| H \|_\alpha, \]
and these inequalities are sharp.
We equip the space of admissible Hamiltonians with the metric and topology induced by any of the norms $\| \cdot \|_\alpha^{a,b}$.
In particular, a sequence $\wH_i$ of admissible Hamiltonians is Cauchy, if and only if the sequence of contact Hamiltonians $H_i$ is Cauchy with respect to the norm $\| \cdot \|_\alpha$ defined by equation~(\ref{eqn:contact-length}).

\section{The contact energy-capacity inequality} \label{sec:energy-capacity-ineq}
The preceding Sections~\ref{sec:contact-geometry}-\ref{sec:symplectization} provide a precise explanation of the statement of Theorem~\ref{thm:energy-capacity-ineq}, and we are ready to give the proof.

\begin{proof}[Proof of Theorem~\ref{thm:energy-capacity-ineq}]
Again consider the symplectization $W = M \times \R$ of $(M,\alpha)$, with symplectic form $\omega = - d (e^\theta \pi_1^* \alpha)$.
Let $a < b$ be two real numbers, and denote $\Khat = K \times [a,b] \subset W$.
Let $c = | h |$.
Equation~(\ref{eqn:lifted-diffeo}) implies that for all $t \in [0,1]$,
	\[ \phi_{\wH}^t (\Khat) \subset \phi_H^t (K) \times [a - c, b + c], \]
and thus $\phi_{\wH}^1 (\Khat) \cap \Khat = \emptyset$.
Let $\rho \colon \R \to [0,1]$ be a smooth cut-off function such that
	\[ \rho(\theta) =
		\left\{
		\begin{array}{ll}
		1 &  \theta \in [a-c, b+c] \\
		0 &  \theta \in \R \setminus (a - c - 1, b + c + 1).
		\end{array}
		\right. \]
By construction and by equation~(\ref{eqn:lifted-diffeo}), we have $\phi_{\rho \wH}^t (x,\theta) = \phi_{\wH}^t (x,\theta)$ for all $x \in M$, $\theta \in [a, b]$, and all times $t \in [0,1]$, and therefore $\phi^1_{\rho \wH}$ also displaces the set $\Khat$.
Thus by the energy-capacity inequality~(\ref{eqn:energy-capacity-ineq}) for compactly supported Hamiltonians \cite{lalonde:gse95},
	\[ 0 < \frac{1}{2} c (\Khat) \leq E (\phi_{\rho \wH}^1). \]
On the other hand, 
	\[ E (\phi_{\rho \wH}^1) \leq \| \rho \wH \|_{\rm Hofer}\leq \| \rho \wH \|_\alpha^{a-c-1, b+c+1} \leq e^{b + c + 1} \| H \|_\alpha,\]
and therefore
\begin{align} \label{eqn:capacity-bound}
	0 < \frac{c (\Khat)}{2 e^{b + 1}} e^{- | h |} \leq \| H \|_\alpha,
\end{align}
proving the theorem.
\end{proof}

It is tempting to think that choosing a smaller value of $b$ above will produce a stronger lower bound in inequality~(\ref{eqn:capacity-bound}).
However, by decreasing $b$, the capacity $c (\Khat)$ decreases as well.
The choice of contact form $\alpha$ affects inequality~(\ref{eqn:capacity-bound}) similarly, as the symplectic form and thus the capacity, the conformal factor, and the Hamiltonian all depend on this choice.
For each of the equivalent norms $\| \cdot \|_\alpha$, defined by a contact form $\alpha$ on $(M,\xi)$, a better lower bound than in inequality~(\ref{eqn:capacity-bound}) is given by
	\[ 0 < \sup_K \left( \frac{c (\Khat)}{2 e^{b + 1}} \right) e^{- | h |} \leq \| H \|_\alpha < \infty, \]
where the supremum is taken over all compact subsets $K$ of $M$ that are displaced by $\phi$, and over all `lifts' $\Khat$ of $K$.
Note that our estimates are valid for any cut-off function $\rho$ as above, and thus one does not need the term $+ 1$ in the exponent above and in inequality~(\ref{eqn:capacity-bound}).
In fact, one could also consider the supremum over all cut-off functions $\rho$ such that $\phi_{\rho \wH}^1 (\Khat) \cap \Khat = \emptyset$.
Conversely, it is also possible to define displacement energy type invariants of subsets of $M$ in this manner.

\section{Topological contact dynamics} \label{sec:topo-contact-dynamics}
In this section, we define the contact distance between two contact isotopies, or between two contact dynamical systems, and introduce topological contact dynamics.

Recall from Section~\ref{sec:contact-geometry} that if $\Phi$ is a smooth contact isotopy of $(M,\xi)$, then a contact form $\alpha$ with $\ker \alpha = \xi$ uniquely determines a generating smooth contact Hamiltonian $H \colon [0,1] \times M \to \R$, and a smooth conformal factor $h \colon [0,1] \times M \to \R$.
If $\alpha' = e^f \alpha$ is any other contact form on $(M,\xi)$, then the corresponding smooth contact Hamiltonian $H'$ and smooth conformal factor $h'$ of the contact isotopy $\Phi$ are given by $H' = e^f H$ and $h' = h + (f \circ \Phi - f)$, respectively.
Also recall that the notation $f \circ \Phi$ stands for the function whose value at $(t,x) \in [0,1] \times M$ is $(f \circ \Phi) (t,x) = f (\phi_t (x))$.
We mention again that when a contact form $\alpha$ on $(M,\xi)$ is chosen, we always assume the contact Hamiltonian and conformal factor of a smooth contact isotopy are determined by this contact form $\alpha$.
Moreover, even when no contact form is selected explicitly, writing $\Phi = \Phi_H$ for a contact isotopy $\Phi$, a contact dynamical system $(\Phi, H, h)$, or calling $H$ the contact Hamiltonian of $\Phi$ and $h$ its conformal factor, implies the choice of a contact form $\alpha$, which is fixed for the remainder of a particular statement or short discussion, unless explicit mention is made to the contrary.

\begin{dfn}[Contact distance]
Let $\alpha$ be a contact form on a contact manifold $(M,\xi)$.
We define the \emph{contact distance} with respect to $\alpha$ between two smooth contact isotopies $\Phi = \Phi_H$ and $\Psi = \Phi_F$ by
\begin{align} \label{eqn:contact-distance}
	d_\alpha (\Phi, \Psi) = d_\alpha (\Phi_H, \Phi_F) = \dbar_M (\Phi_H, \Phi_F) + | h - f | + \| H - F \|_\alpha.
\end{align}
\end{dfn}

Recall from Section~\ref{sec:contact-geometry} that a $\dbar_M$-Cauchy sequence $\Phi_i$ of contact isotopies converges to a continuous isotopy $\Phi = \{ \phi_t \}_{0 \le t \le 1}$, a $\| \cdot \|_\alpha$-Cauchy sequence $H_i$ of contact Hamiltonians converges to an $L^\oneinfty$-function $H \colon [0,1] \times M \to \R$, and a $| \cdot |$-Cauchy sequence of conformal factors $h_i$ uniformly converges to a continuous function $h \colon [0,1] \times M \to \R$.

\begin{dfn}[Topological contact dynamical system] \label{dfn:topo-contact-dynamics}
Let $\alpha$ be a contact form on a contact manifold $(M,\xi)$.
A triple $(\Phi, H, h)$ is a \emph{topological contact dynamical system} with respect to the contact form $\alpha$, if there exists a sequence $(\Phi_{H_i}, H_i, h_i)$ of smooth contact dynamical systems, such that as $i \to \infty$, the sequence $\Phi_{H_i}$ $C^0$-converges to the continuous isotopy $\Phi \in \PHomeo (M)$, the sequence $H_i$ of smooth contact Hamiltonians satisfies $\| H - H_i \|_\alpha \to 0$, and the sequence $h_i$ of smooth conformal factors converges uniformly to the continuous function $h$.
The $L^\oneinfty$-function $H \colon [0,1] \times M \to \R$ is called a \emph{topological contact Hamiltonian} with \emph{topological contact isotopy} $\Phi$ and \emph{topological conformal factor} $h \colon [0,1] \times M \to \R$.
The space of topological contact dynamical systems is denoted by $\TC (M,\alpha)$, and the space of topological contact isotopies by $\PHomeo (M,\xi)$.
By a slight abuse of notation, we denote the natural extension $\overline{d}_\alpha$ of the contact metric defined by equation~(\ref{eqn:contact-distance}) also by $d_\alpha$, and call it the \emph{contact metric} on the space $\TC (M,\alpha)$.
\end{dfn}

By definition a topological contact dynamical system represents an equivalence class of Cauchy sequences of smooth contact dynamical systems with respect to the contact distance $d_\alpha$ defined above.
The metric $d_\alpha$ does depend on the choice of contact form $\alpha$, however, the different choices of contact form lead to equivalent metrics.
In particular, the collection of Cauchy sequences with respect to $d_\alpha$ as well as the topology induced by $d_\alpha$ only depend on the contact structure $\xi$.

\begin{lem} \label{lem:alpha-indep}
Suppose $\alpha$ and $\alpha' = e^g \alpha$ are two contact forms on $(M,\xi)$.
Then there exist constants $m (g)$ and $M (g)$ that depend only on the function $g$, such that
	\[ m (g) \cdot d_\alpha (\Phi, \Psi) \le d_{\alpha'} (\Phi, \Psi) \le M (g) \cdot d_\alpha (\Phi, \Psi) \]
for any smooth contact isotopies $\Phi$ and $\Psi$ of $(M,\xi)$.
Moreover, a topological contact dynamical system $(\Phi, H, h)$ with respect to $\alpha$ is transformed to the topological contact dynamical system
	\[ (\Phi, e^g H, h + (g \circ \Phi - g)) \]
with respect to $e^g \alpha$, and the space $\PHomeo (M,\xi)$ of topological contact isotopies is independent of the choice of contact form $\alpha$.
\end{lem}

\begin{proof}
Let $H$ be the smooth contact Hamiltonian generating the isotopy $\Phi = \{ \phi_t \}$ with respect to $\alpha$.
The smooth function $H' = e^g H$ then generates the same isotopy with respect to the contact form $\alpha'$.
Furthermore, since $\phi_t^* \alpha = e^{h_t} \alpha$, we have
	\[ \phi_t^* \alpha' = e^{h_t + (g \circ \phi_t - g)} \alpha'. \]
Similarly, if $F$ and $f$ are the smooth contact Hamiltonian and conformal factor of the isotopy $\Psi$ with respect to $\alpha$, then $e^g F$ and $f + (g \circ \Psi - g)$ are the smooth contact Hamiltonian and conformal factor of $\Psi$ with respect to $\alpha'$.
Then
\begin{align*}
	d_{\alpha'} (\Phi, \Psi) = {} & \dbar_M (\Phi, \Psi) + | h + g \circ \Phi - f - g \circ \Psi | + \| e^g H - e^g F \|_{\alpha'} \\
	\leq {} & \dbar_M (\Phi, \Psi) + | h - f | + | g \circ \Phi - g \circ \Psi | + C (g) \cdot \| H - F \|_\alpha \\
	\leq {} & \dbar_M (\Phi, \Psi) + | h - f | + L (g) \cdot d_M (\Phi, \Psi) + C (g) \cdot \| H - F \|_\alpha, \\
	\leq {} & \max (1 + L (g), C (g)) \cdot d_\alpha (\Phi, \Psi),
\end{align*}
where the positive constant $C (g)$ is as in Lemma~\ref{lem:equiv-norms}, and $L (g) > 0$ is a Lipschitz constant.
Reversing the roles of $\alpha$ and $\alpha'$ proves the other inequality.
In particular, a sequence of smooth contact isotopies is Cauchy with respect to $\alpha$ if and only if it is Cauchy with respect to $\alpha'$.
The formula for the transformed topological contact dynamical system follows from the above computations.
\end{proof}

Examples of non-smooth topological contact dynamical systems are given in \cite{ms:tcd2}.

\begin{thm}[Uniqueness of topological Hamiltonian isotopy and conformal factor] \label{thm:uniqueness}
Fix a contact form $\alpha$ on a contact manifold $(M,\xi)$.
Then the topological Hamiltonian isotopy and the topological conformal factor of a topological contact Hamiltonian are unique.
More precisely, if $(\Phi, H, h)$ and $(\Psi, H, g)$ are two topological contact dynamical systems with the same topological contact Hamiltonian, then $\Phi = \Psi$ and $h = g$.
\end{thm}

In other words, a topological contact Hamiltonian $H$ uniquely determines the topological contact dynamical system $(\Phi, H, h) \in \TC (M,\alpha)$.
The proof is given in Section~\ref{sec:uniqueness-theorems}.
Examples showing that all of the convergence hypotheses in Definition~\ref{dfn:topo-contact-dynamics} and in the theorem are necessary are produced in Section~\ref{sec:examples}.
That the topological contact isotopy in turn also uniquely determines the topological contact dynamical system is proved in the sequel \cite{ms:tcd3}.

Given a topological contact Hamiltonian $H$, we denote the unique corresponding topological contact isotopy by $\Phi_H$, and the unique corresponding topological conformal factor by the lower case Roman letter $h$.
As in the smooth case, writing $\Phi = \Phi_H$ for a topological contact isotopy, and calling $h$ its topological conformal factor, or writing $(\Phi, H, h)$ for a topological contact dynamical system, involves the explicit or implicit selection of a contact form $\alpha$ with kernel $\xi$.
By Lemma~\ref{lem:alpha-indep}, if $e^f \alpha$ is another contact form on $(M,\xi)$, then the topological contact dynamical system $(\Phi, H, h)$ with respect to $\alpha$ is transformed to the topological contact dynamical system $(\Phi, e^f H, h + (f \circ \Phi - f))$ with respect to $e^f \alpha$.

By the above uniqueness theorem, given two topological contact Hamiltonians $H$ and $F$, we can define the functions $\Hbar \# F$ and $\holine \# f$ by
\begin{align}
	(\Hbar \# F)_t & = e^{- h_t} \cdot \left( (F_t - H_t) \circ \phi_H^t \right), \label{eqn:gp-str-ham} \\
	(\holine \# f)_t & = - h_t \circ (\phi^t_H)^{-1} \circ \phi^t_F + f_t, \label{eqn:gp-str-fctn}
\end{align}
where $\{ \phi_H^t \}$ and $\{ \phi_F^t \}$ are the unique topological contact isotopies corresponding to the topological contact Hamiltonians $H$ and $F$, respectively, and similarly, $h$ and $f$ are the corresponding unique topological conformal factors.
These operations extend the group structure in Lemmas~\ref{lem:conformal-calculus} and \ref{lem:ham-gp-str}.
A group structure on the space $\TC (M,\alpha)$ of topological contact dynamical systems $(\Phi, H, h)$ can then be defined by
\begin{align} \label{eqn:tcds-gp-str}
	(\Phi_H, H, h)^{-1} \circ (\Phi_F, F, f) = (\Phi_H^{-1} \circ \Phi_F, \Hbar \# F, \holine \# f).
\end{align}
Formulas~(\ref{eqn:gp-str-ham})--(\ref{eqn:tcds-gp-str}) determine the group operations completely, by inserting the identity $(\id,0,0)$ for the topological contact dynamical system $(\Phi_F, F, f)$, and then the inverse $(\Phi_H, H, h)^{-1}$ for $(\Phi_H, H, h)$.
In fact, this group structure on $\TC (M,\alpha)$ is well-defined even without the uniqueness theorem at hand, while defining the group structure on the space of topological contact Hamiltonians given by equation~(\ref{eqn:gp-str-ham}) does require Theorem~\ref{thm:uniqueness}.
As in the smooth case, equation~(\ref{eqn:gp-str-fctn}) by itself does not make sense, and is well-defined only if two topological contact Hamiltonians $H$ and $F$, or as we will see below, two topological contact isotopies, are given, and $h$ and $f$ are the corresponding topological conformal factors.

\begin{thm} \label{thm:topo-group}
The metric space $\TC (M,\alpha)$ of topological contact dynamical systems $(\Phi_H, H, h)$ of $(M,\alpha)$ forms a topological group with identity $(\id, 0, 0)$ under the operation $\circ$ defined in equation~(\ref{eqn:tcds-gp-str}).
The group $\C (M,\alpha)$ of smooth contact dynamical systems forms a topological subgroup.
\end{thm}

We prove this theorem in Section~\ref{sec:topo-group}.
As an immediate corollary, we obtain

\begin{cor}
The group structure on $\TC (M,\alpha)$ induces via the projections group structures on the space $\PHomeo (M,\xi)$ of topological contact isotopies, on the space $L_\alpha^\oneinfty ([0,1] \times M)$ of topological contact Hamiltonians, and on the set of time-one maps of topological contact isotopies.
\end{cor}

In fact, in the case of isotopies and their time-one maps, the group structure is the usual one defined by composition of homeomorphisms of $M$, and one obtains topological subgroups of $\PHomeo (M)$ and $\Homeo (M)$, respectively.

\begin{dfn}[Contact homeomorphism]
A homeomorphism $\phi$ of $M$ is called a \emph{contact homeomorphism} if it is the time-one map of a topological contact isotopy.
The group of contact homeomorphisms is denoted by $\Homeo (M, \xi)$.
\end{dfn}

\begin{dfn}[Topological automorphism of the contact structure]
A homeomorphism $\phi$ of $M$ is a \emph{topological automorphism} of the contact structure $\xi$, if there exists a sequence of contact diffeomorphisms $\phi_i \in \Diff (M,\xi)$ with $\phi_i^* \alpha = e^{h_i} \alpha$, such that the sequence $\phi_i$ $C^0$-converges to the homeomorphism $\phi$, and the sequence of smooth conformal factors $h_i$ converges uniformly to a continuous function $h$.
The set of topological automorphisms is denoted by $\Aut (M,\xi)$, and the function $h \in C^0 (M)$ is called the \emph{topological conformal factor} of the automorphism $\phi$ with respect to the contact form $\alpha$.
\end{dfn}

\begin{thm}[Uniqueness of topological conformal factor of automorphism] \label{thm:unique-topo-conformal-factor}
The topological conformal factor of an automorphism $\phi \in \Aut (M,\xi)$ is uniquely determined by the homeomorphism $\phi$ and the contact form $\alpha$.
That is, suppose there exist two sequences $\phi_i$ and $\psi_i \in \Diff (M,\xi)$ with $\phi_i^* \alpha = e^{h_i} \alpha$ and $\psi_i^* \alpha = e^{g_i} \alpha$, such that both sequences $\phi_i$ and $\psi_i$ $C^0$-converge to the homeomorphism $\phi$, and the sequences $h_i$ and $g_i$ uniformly converge to continuous functions $h$ and $g$, respectively.
Then $h = g$.
\end{thm}

Equivalently, if $\phi = \id$, then we must have $h = 0$.
That is, if the sequence $\phi_i$ of contact diffeomorphisms with $\phi_i^* \alpha = e^{h_i} \alpha$ $C^0$-converges to the identity, and the sequence $h_i$ converges uniformly to a continuous function $h$, then we must have $h = 0$.
See Section~\ref{sec:automorphisms} for the proofs.

\begin{cor}[Uniqueness of topological conformal factor of isotopy]  \label{cor:unique-topo-conformal-factor-iso}
The topological conformal factor of a topological contact isotopy $\Phi$ is uniquely determined by $\Phi$ and the contact form $\alpha$.
That is, if $(\Phi, H, h)$ and $(\Phi, F, f)$ are two topological contact dynamical systems with the same topological contact isotopy, then $h = f$.
\end{cor}

\begin{proof}
Each time-$t$ map $\phi_H^t$ is contained in $\Aut (M,\xi)$.
By Proposition~\ref{thm:unique-topo-conformal-factor}, for each $t$ the continuous function $h_t$ is uniquely determined by $\phi_H^t$.
\end{proof}

\begin{pro} \label{pro:gp-str-auto}
The set $\Aut (M,\xi)$ forms a subgroup of $\Homeo (M)$, and it contains as subgroups the groups $\Diff (M,\xi)$ and $\Homeo (M,\xi) \subseteq \Aut (M,\xi)$.
If $\phi$ and $\psi \in \Aut (M,\xi)$ are topological automorphisms with topological conformal factors $h$ and $g$, respectively, then the topological conformal factors of $\phi \circ \psi$ and $\phi^{-1}$ are $h \circ \psi + g$ and $-h \circ \phi^{-1}$, respectively.
\end{pro}

The proof is obvious from the definitions.
See Lemma~\ref{lem:conformal-calculus} for the last part.
Recall that if $\alpha$ and $e^f \alpha$ are two contact forms on $(M,\xi)$ and $\phi^* \alpha = e^h \alpha$, then $\phi^* (e^f \alpha) = e^{h + (f \circ \phi - f)} (e^f \alpha)$.

\begin{pro} \label{pro:indep-auto-gp}
The automorphism group $\Aut (M,\xi)$ does not depend on the choice of contact form $\alpha$.
More precisely, suppose $\alpha$ is a contact form with $\ker \alpha = \xi$, and there exists a sequence of contact diffeomorphisms $\phi_i$ with $\phi_i^* \alpha = e^{h_i} \alpha$, such that the sequence $\phi_i$ $C^0$-converges to a homeomorphism $\phi$, and the sequence of conformal factors $h_i$ converges uniformly to a continuous function $h$.
Further suppose that $e^f \alpha$ is any other contact form on $(M,\xi)$.
Then $\phi$ is also a topological automorphism with respect to the contact form $e^f \alpha$, with topological conformal factor $h + (f \circ \phi - f)$, i.e.\ the conformal factors $h_i + (f \circ \phi_i - f)$ converge to the continuous function $h + (f \circ \phi - f)$ uniformly.
\end{pro}

The proof of the proposition is again immediate.

\begin{thm}[Transformation law] \label{thm:transformation-law}
Let $(\Phi_H, H, h)$ be a topological contact dynamical system, and $\phi \in \Aut (M,\xi)$ be a topological automorphism of the contact structure $\xi$, with topological conformal factor $g$.
Then
\begin{align} \label{eqn:transformation-law}
	(\phi^{-1} \circ \Phi_H \circ \phi, e^{- g} (H \circ \phi), h \circ \phi + g - g \circ \phi^{-1} \circ \Phi_H \circ \phi)
\end{align}
is a topological contact dynamical system.
\end{thm}

See Section~\ref{sec:topo-group} for the proof.
Recall that with the same notation as in the rest of the paper, the conformal factor in the theorem is the continuous function
	\[ (h \circ \phi + g - g \circ \phi^{-1} \circ \Phi_H \circ \phi) (t,x) = h (t,\phi (x)) + g (x) - g (\phi^{-1} (\phi_H^t (\phi (x)))) \]
on $[0,1] \times M$, and similarly for the topological contact Hamiltonian $e^{- g} (H \circ \phi)$.

\begin{cor}[Normality] \label{cor:normal}
The group of contact homeomorphisms is a normal subgroup of the topological automorphism group of the contact structure,
	\[ \Homeo(M,\xi) \trianglelefteq \Aut(M,\xi) \subseteq \Homeo (M). \]
\end{cor}

\begin{pro}[Path-connectedness]
Let $\Phi = \{ \phi_t \}$ be a topological contact isotopy.
Then each time-$t$ map $\phi_t$ is a contact homeomorphism.
In particular, $\Homeo (M,\xi)$ is path-connected in the $C^0$-topology.
\end{pro}

The proof of the proposition is a consequence of Lemma~\ref{lem:rep} in the next section.
Other topological properties of $\Homeo (M,\xi)$ are studied in the sequel \cite{ms:tcd2}.
The analogous theorems in the Hamiltonian and strictly contact case are stated and proved in \cite{mueller:ghh07, banyaga:ugh11}.

Using $C^{1,1}$-functions instead of smooth contact Hamiltonians in Definition~\ref{dfn:topo-contact-dynamics} leads to the same notion of topological contact dynamics, and the proof is almost the same as in the case of Hamiltonian dynamical systems in \cite{mueller:ghh07}.
Recall that a time-dependent continuous vector field $X$ is uniquely integrable, provided $X (t,\cdot)$ is (locally) Lipschitz  independent of time $t \in [0,1]$.

\begin{thm}
Suppose $H \colon [0,1] \times M \to \R$ is a continuous function that is continuously differentiable in the variable $x \in M$, the one-form $dH$ is continuous in $t$, and the time-dependent vector field $X_H$ is uniquely integrable.
Denote by $\Phi_H$ the continuous isotopy generated by $X_H$, and by $h \colon [0,1] \times M \to \R$ the continuous function defined by equation~(\ref{eqn:formula-conformal-factor}).
Then $H$ is a topological contact Hamiltonian with topological contact isotopy $\Phi_H$ and topological conformal factor $h$.
\end{thm}

\begin{proof}
The given function $H$ can be approximated by a sequence of smooth contact Hamiltonians $H_i \colon [0,1] \times M \to \R$ such that $H_i^t \to H$ and $dH_i^t \to dH_t$ uniformly in $x \in M$ and $t \in [0,1]$.
Thus $\| H - H_i \|_\alpha \to 0$ as $i \to \infty$, and the Lipschitz vector fields $X_{H_i}$ converge to $X_H$ uniformly over $t \in [0,1]$ and $x \in M$, cf.\ equation~(\ref{eqn:contact-ham}).
Therefore the flows $\Phi_{H_i}$ converge uniformly to $\Phi_H$ by the standard continuity theorem in the theory of ordinary differential equations, and thus also in the $C^0$-metric.
In particular $h_i \to h$ uniformly over $t \in [0,1]$ and $x \in M$ by equation~(\ref{eqn:formula-conformal-factor}).
Thus $(\Phi_{H_i}, H_i, h_i)$ converges to $(\Phi_H, H, h)$ in the contact metric $d_\alpha$.
\end{proof}

We note that the proof does not invoke the Uniqueness Theorem~\ref{thm:uniqueness}.

\begin{dfn}[Admissible topological Hamiltonian dynamical system]
Let $(M,\xi)$ be a contact manifold, and $\alpha$ a contact form with $\ker \alpha = \xi$.
Denote by $W = M \times \R$ the corresponding symplectization with symplectic form $\omega = - d (e^\theta \pi_1^* \alpha)$.
A pair $(\Phihat, \wH)$ is an \emph{admissible topological Hamiltonian dynamical system} of $(W,\omega)$ if there exists a sequence of smooth admissible Hamiltonian isotopies $\Phi_{\wH_i}$ that $C^0$-converges to the continuous isotopy $\Phihat = \{ \phihat_t \}$ of $W$, and $\| \wH - \wH_i \|_\alpha^{a,b} \to 0$ for a (and thus any) $K_{a,b} = M \times [a,b] \subset M \times \R$.
The function $\wH \colon [0,1] \times W \to \R$ is called an \emph{admissible topological Hamiltonian} with \emph{admissible topological Hamiltonian isotopy} $\Phihat$.
\end{dfn}

Given a time-dependent function $H$, a continuous isotopy of homeomorphisms $\Phi$, and a time-dependent continuous function $h$, define as in Section~\ref{sec:symplectization} the function $\wH (t,x,\theta) = e^\theta H (t,x)$ and the isotopy $\Phihat = \{ \phihat_t \}$ on $[0,1] \times M \times \R$ by $\phihat_t (x,\theta) = (\phi_t (x), \theta - h_t (x))$.
By construction, $(\Phi, H, h)$ is a topological contact dynamical system with respect to the contact form $\alpha$, if and only if $(\Phihat, \wH)$ is an admissible topological Hamiltonian dynamical system of the symplectization of $(M,\alpha)$.
Thus all of the definitions and results in topological contact dynamics have analogs in admissible topological Hamiltonian dynamics of symplectizations, and the proofs are verbatim the same.

A \emph{topological strictly contact dynamical system} $(\Phi, H, 0)$ is by definition the limit of a $d_\alpha$-convergent sequence of smooth strictly contact dynamical systems $(\Phi_i, H_i, 0)$ \cite{banyaga:ugh11}.
Topological strictly contact dynamical systems form a topological subgroup of the group of topological contact dynamical systems.
The constructions in this article generalize those in \cite{banyaga:ugh11} in the strictly contact case, taking into account the added complications of non-trivial conformal factors in several places in our definitions and proofs.

As in the Hamiltonian case in \cite{mueller:ghh07} or Chapter~3 in \cite{mueller:ghh08}, it is straightforward to define \emph{compactly supported topological dynamical systems} of open contact manifolds $(M,\xi)$, provided only that $\xi$ is coorientable.
If $M$ is open, one restricts to homeomorphisms, isotopies, and functions on $[0,1] \times M$ that are compactly supported in the interior of $M$, or in other words, have compact support and are trivial near the boundary of $M$, and to Cauchy sequences that are supported in a compact subset $K \subset \Int \, M$ independently of the index $i$ of the sequence.
With these modifications, all the definitions and proofs in this paper hold for open contact manifolds.
The rigidity theorems in Section~\ref{sec:automorphisms} are local statements, and thus it suffices in those cases to restrict to homeomorphisms that are the identity on the boundary, and instead of compact support require only convergence on compact subsets.

Following the ideas presented in this article, it is a straightforward task to extend the notion of a topological Hamiltonian dynamical system to other types of non-compact symplectic manifolds that appear for example in the context of symplectic field theory.
This is the case for instance when $(W,\omega)$ is a \emph{symplectic manifold with cylindrical ends}.
Details may be published elsewhere in this series of papers.

\section{The uniqueness theorems} \label{sec:uniqueness-theorems}
In this section, we prove several uniqueness and rigidity results, culminating in the proof of Theorem~\ref{thm:uniqueness}.
These results are inspired by similar theorems for compactly supported Hamiltonians on symplectic manifolds, see \cite{mueller:ghh07} or sections 2.2 and 2.3 in \cite{mueller:ghh08}.
As above, let $(M,\xi)$ be a contact manifold with a contact form $\alpha$, and let $W = M \times \R$ denote the symplectization of $(M,\alpha)$ with symplectic form $\omega = - d (e^\theta \pi_1^* \alpha)$.

\begin{pro} \label{pro:h-z-uniqueness}
Let $\Phi_{H_i}$ be a sequence of smooth contact isotopies of $M$, $\Phi_H$ be another smooth contact isotopy, and $\phi \colon M \to M$ be a function.
Assume
\begin{enumerate}
	\item [(i)] $\| \Hbar \# H_i \|_\alpha \to 0$, as $i \to \infty$,
	\item [(ii)] $\phi_{H_i}^1 \to \phi$ uniformly, as $i \to \infty$, and
	\item [(iii)] $| h_i | \le c$ for some constant $c \in \R$ independently of $i$,
\end{enumerate}
where $h_i \colon [0,1] \times M \to \R$ is given by $(\phi_{H_i}^t)^* \alpha = e^{h_i (t,\cdot)} \alpha$.
Then $\phi = \phi_H^1$.
\end{pro}

Since by hypothesis (iii), the sequence $| h_i |$ is bounded independently of $i$, the constants $C (h_i)$ in Lemma~\ref{lem:equiv-norms} can be chosen independently of $i$ as well.
Therefore hypothesis (i) in the proposition is equivalent to the assumption $\| H - H_i \|_\alpha \to 0$ as $i \to \infty$.
The same observation applies in the remainder of the section, and we do not need to distinguish between convergence of the sequence $\Hbar \# H_i \to 0$ and $H_i \to H$ with respect to the distance induced by the norm $\| \cdot \|_\alpha$.

\begin{proof}
Because $\phi$ is the uniform limit of continuous maps $\phi_{H_i}^1$, it must be continuous.
Suppose to the contrary that $\phi \not= \phi_H^1$.
Then there exists a compact ball $B \subset M$ such that $((\phi_H^1)^{-1} \circ \phi) (B) \cap B = \emptyset$.
By hypothesis (ii), $\phi_{H_i}^1 \to \phi$ uniformly, and thus $((\phi_H^1)^{-1} \circ \phi_{H_i}^1) (B) \cap B = \emptyset,$ for all sufficiently large $i$.
But then by Theorem~\ref{thm:energy-capacity-ineq},
	\[ \| \Hbar \# H_i \|_\alpha \ge C e^{- \left| h_i - h \circ \Phi_H^{-1} \circ \Phi_{H_i} \right|} \ge C e^{- c - | h |} > 0, \]
where $h$ is defined by $(\phi_H^t)^* \alpha = e^{h (t,\cdot)} \alpha$.
This contradicts hypothesis (i).
\end{proof}

\begin{cor} \label{cor:unique-isotopy}
Let $\Phi_{H_i}$ be a sequence of smooth contact isotopies of $M$, $\Phi_H$ be another smooth contact isotopy, and $\Phi$ be an isotopy of functions $\phi_t \colon M \to M$.
Assume
\begin{enumerate}
	\item [(i)] $\| \Hbar \# H_i \|_\alpha \to 0$, as $i \to \infty$,
	\item [(ii)] $\Phi_{H_i} \to \Phi$ uniformly, as $i \to \infty$, and
	\item [(iii)] $| h_i | \le c$ for some constant $c \in \R$ independently of $i$,
\end{enumerate}
where $h_i \colon [0,1] \times M \to \R$ is given by $(\phi_{H_i}^t)^* \alpha = e^{h_i (t,\cdot)} \alpha$.
Then $\Phi = \Phi_H$.
\end{cor}

In order to give the proof, we need the smooth version of the next lemma.
The proof is straightforward and thus omitted.

\begin{lem} \label{lem:rep}
Let $(\Phi_H, H, h)$ be a smooth (or topological) contact dynamical system.
For $s \in [0,1]$, the reparametrization $\Phi_{H^s} = \{ \phi_{H^s}^t \} = \{ \phi_H^{s t} \}$ is also a smooth (or topological) contact isotopy, with time-one map $\phi_H^s$, smooth (or topological) contact Hamiltonian $H^s$, and smooth (or topological) conformal factor $h^s$, where the Hamiltonian and conformal factor are given by
	\[ H^s (t,x) = s H (s t,x) \quad \text{and} \quad h^s (t,x) = h (s t,x). \]
\end{lem}

More general reparametrizations, where the map $t \mapsto s t$ is replaced by a smooth function $\zeta \colon [0,1] \to [0,1]$, will be considered in Sections~\ref{sec:examples} and \ref{sec:bi-invariant-metric} as well as in the sequels \cite{ms:tcd2, ms:tcd3}.

\begin{proof}[Proof of Corollary~\ref{cor:unique-isotopy}]
Suppose the contrary that $\Phi \not= \Phi_H$, i.e.\ there exists $s \in (0,1]$ such that $\phi_s \not= \phi_H^s$.
By Lemma~\ref{lem:rep}, the smooth contact dynamical systems $(\Phi_{H_i^s}, H_i^s, h_i^s)$, the smooth contact isotopy $\Phi_{H^s}$, and the function $\phi_s$, together satisfy the hypothesis of Proposition~\ref{pro:h-z-uniqueness}.
Thus reparametrizing with the function $t \mapsto s t$, we may assume $s = 1$.
Applying Proposition~\ref{pro:h-z-uniqueness} yields a contradiction, hence the proof.
\end{proof}

\begin{pro} \label{pro:unique-conformal-factor}
Let $\Phi_{H_i}$ be a sequence of contact isotopies on $M$, $\Phi_H$ be another smooth contact isotopy, and $g \colon [0,1] \times M \to \R$ be a function.
Assume
\begin{enumerate}
	\item [(i)] $\| \Hbar \# H_i \|_\alpha \to 0$, as $i \to \infty$, and
	\item [(ii)] $| g - h_i | \to 0$ as $i \to \infty$,
\end{enumerate}
where $h_i \colon [0,1] \times M \to \R$ is given by $(\phi_{H_i}^t)^* \alpha = e^{h_i (t,\cdot)} \alpha$, and similarly $(\phi_H^t)^* \alpha = e^{h (t,\cdot)} \alpha$.
Then $g = h$.
\end{pro}

\begin{proof}
Again $g$ must be continuous since it is the uniform limit of continuous functions $h_i$.
Suppose the contrary that $g \not= h$.
Then there exists $s \in (0,1]$, $B \subset M$ a sufficiently small compact ball, and $\epsilon > 0$, such that $| g (s,x) - h (s,x) | > 2 \epsilon$ for all $x \in B$.
Recall that 
	\[ \phi_{\wH_i}^s \circ (\phi_{\wH}^s)^{-1} = \left( \phi_{H_i}^s \circ (\phi_H^s)^{-1}, \theta - (h_i (s,\cdot) - h (s,\cdot)) \circ (\phi_H^s)^{-1} \right). \]
Thus if $\Khat = \phi_H^s (B) \times [-\epsilon,\epsilon]$, then
	\[ \left( \phi_{\wH_i}^s \circ (\phi_{\wH}^s)^{-1} \right) (\Khat) \cap \Khat = \emptyset, \]
for all sufficiently large $i$.
Hypothesis (ii) implies that the numbers $| h_i |$ are bounded by a constant $c$ independently of $i$.
Arguing as in the proof of Theorem~\ref{thm:energy-capacity-ineq} and Corollary~\ref{cor:unique-isotopy}, choose a cut-off function $\rho$, and apply the energy-capacity inequality.
From hypothesis (i) one then derives a contradiction.
\end{proof}

Note that the corresponding isotopies $\Phi_{H_i}$ being uniformly Cauchy is not necessary for the proof.
Displacement of the set $\Khat$, and being able to choose the cut-off function $\rho$ independently of $i$, is guaranteed by hypothesis (ii).

Combining Corollary~\ref{cor:unique-isotopy} and Proposition~\ref{pro:unique-conformal-factor}, we obtain the main uniqueness theorem of this article.

\begin{cor} \label{cor:cds-unique}
Let $\Phi_{H_i}$ be a sequence of smooth contact isotopies of $M$, $\Phi_H$ be another smooth contact isotopy, $\Phi$ be an isotopy of functions $\phi^t \colon M \to M$, and $g \colon [0,1] \times M \to \R$ be a function.
Assume
\begin{enumerate}
	\item [(i)] $\| \Hbar \# H_i \|_\alpha \to 0$, as $i \to \infty$,
	\item [(ii)] $\Phi_{H_i} \to \Phi$ uniformly, as $i \to \infty$, and
	\item [(iii)] $| g - h_i | \to 0$ as $i \to \infty$,
\end{enumerate}
where $h_i \colon [0,1] \times M \to \R$ is given by $(\phi_{H_i}^t)^* \alpha = e^{h_i (t,\cdot)} \alpha$, and similarly $(\phi_H^t)^* \alpha = e^{h (t,\cdot)} \alpha$.
Then $\Phi = \Phi_H$ and $g = h$.
\end{cor}

\begin{lem} \label{lem:isotopy-equivalent}
Let $(\Phi_i, H_i, h_i) \in \C (M,\alpha)$ be a sequence of smooth contact dynamical systems, converging with respect to the contact metric $d_\alpha$ to the topological contact dynamical system $(\Phi, H, h) \in \TC (M,\alpha)$.
Then the following statements are all equivalent.
\begin{enumerate}
	\item [(i)] Suppose $(\Psi_i, F_i, f_i)$ is another sequence of smooth contact dynamical systems that converges with respect to the contact metric $d_\alpha$ to the topological contact dynamical system $(\Psi, F, f) \in \TC (M,\alpha)$.
	If $H = F$, then $\Phi = \Psi$, and $h = f$.
	\item [(ii)] If $H$ is smooth, then $\Phi$ is a smooth isotopy, and in fact, $\Phi = \Phi_H$ is the smooth contact isotopy generated by the smooth function $H$ in the sense of equation~(\ref{eqn:contact-ham}).
	Moreover, the function $h$ is smooth, and equals the smooth conformal factor of the smooth contact isotopy $\Phi_H$, i.e.\ $(\phi_H^t)^* \alpha = e^{h (t,\cdot)} \alpha$.
	\item [(iii)] If $H = 0$, then $\Phi = \id$, and $h = 0$.
\end{enumerate}
\end{lem}

\begin{proof}
To see that (i) implies (ii), choose the sequence $F_i = H$.
That (ii) implies (iii) is obvious, since the zero Hamiltonian is a smooth function.
We prove that (iii) in turn implies (i).
By Theorem~\ref{thm:topo-group}, the sequence of smooth contact Hamiltonians $\Hbar_i \# F_i$ converges to the zero Hamiltonian, the sequence of isotopies $\Phi_{H_i}^{-1} \circ \Phi_{F_i}$ converges to $\Phi^{-1} \circ \Psi$ in the $C^0$-metric, and the sequence of conformal factors of $\Phi_{H_i}^{-1} \circ \Phi_{F_i}$ converges to the continuous function $g - h \circ \Phi^{-1} \circ \Psi$ uniformly.
Then by (iii), $\Phi^{-1} \circ \Psi = \id$, and $g - h = g - h \circ \Phi^{-1} \circ \Psi = 0$.
\end{proof}

Although not stated explicitly in \cite{mueller:ghh07}, an analogous lemma also holds for Hamiltonian dynamical systems.
See Section~\ref{sec:sequels} for the converse statement.

\begin{proof}[Proof of Theorem~\ref{thm:uniqueness}]
Combine Corollary~\ref{cor:cds-unique} and Lemma~\ref{lem:isotopy-equivalent}.
\end{proof}

The uniqueness results of this paper have a number of immediate consequences for topological contact dynamical systems that resemble well-known results in the smooth case.
As a demonstration, we prove two lemmas.
See also Section~\ref{sec:smooth-conformal-factors} and \cite{ms:hvf12, ms:tcd3}.

\begin{lem}
Let $(\Phi, H, h)$ be a topological contact dynamical system, and suppose that the function $H$ is autonomous, and $H \circ \phi_H^t = e^{h_t} H$ for all $t$.
Then $\Phi = \{ \phi_t \}$ is a one-parameter subgroup of $\Aut (M,\xi)$.
\end{lem}

\begin{proof}
Fix $s \in [0,1]$.
Since $\phi_H^s \in \Aut (M,\xi)$, the isotopy $\{ \phi_H^{t + s} \circ (\phi_H^s)^{-1} \}$ is a topological contact isotopy.
By hypothesis, $H_t = H_{t + s}$ for all $0 \le t \le 1$, therefore this isotopy coincides with the topological contact isotopy $\{ \phi_H^t \}$ by Theorem~\ref{thm:uniqueness}.
Similarly, $H = e^{- h_s} (H \circ \phi_H^s)$, so that the topological contact isotopies $\{ \phi_H^t \}$ and $\{ (\phi_H^s)^{-1} \circ \phi_H^{t + s} \}$ coincide.
Thus $\phi_H^t \circ \phi_H^s = \phi_H^{t + s} = \phi_H^s \circ \phi_H^t$ for all $0 \le s, t \le 1$, and $\Phi$ is a one-parameter subgroup.
\end{proof}

\begin{dfn}[Basic function]
A (not necessarily differentiable) function $H \colon M \to \R$ is \emph{basic} if it is invariant under the Reeb flow, i.e.\ $H (\phi_R^s (x)) = H (x)$ for all $x \in M$, and all $0 \le s \le 1$, where $\{ \phi_R^s \}$ denotes the Reeb flow.
A time-dependent function $H \colon [0,1] \times M \to \R$ is basic if the function $H_t$ is basic at each time $0 \le t \le 1$.
\end{dfn}

\begin{lem}
This definition coincides with the usual definition of a smooth basic function if $H$ is continuously differentiable in the Reeb direction.
\end{lem}

\begin{proof}
The claim follows immediately from the identities
	\[ (\phi_R^s)^* (R_\alpha . H_t) = (\phi_R^s)^* (\mL_{R_\alpha} H_t) = \frac{d}{ds} ((\phi_R^s)^* H_t) = \frac{d}{ds} (H_t \circ \phi_R^s). \qedhere \]
\end{proof}

\begin{lem}
If $H$ is a basic topological contact Hamiltonian with topological contact isotopy $\Phi_H$, then $\phi_H^t$ commutes with the Reeb flow $\{ \phi_R^s \}$ of $\alpha$ for all times $s$ and $t$.
\end{lem}

In the present language, this result first appeared in \cite{banyaga:ugh11} under the hypothesis that $(\Phi_H, H, 0)$ is a topological strictly contact dynamical system.

\begin{proof}
Fix a time $s$.
By hypothesis, the topological contact Hamiltonians $H$ and $H \circ \phi_R^s$ coincide.
Thus by Theorem~\ref{thm:transformation-law}, and by uniqueness of the topological contact isotopy, $(\phi_R^s)^{-1} \circ \phi_H^t \circ \phi_R^s = \phi_H^t$.
\end{proof}

Appropriate local versions of the uniqueness results in this paper hold as well \cite{ms:tcd3}.

\section{Examples of divergent sequences} \label{sec:examples}
A topological contact dynamical system $(\Phi, H, h)$ is determined by three Cauchy sequences, namely a sequence of smooth contact isotopies $\Phi_i$, the corresponding sequence $H_i$ of smooth time-dependent contact Hamiltonian functions, and the sequence of time-dependent conformal factors $h_i$ of the isotopies $\Phi_i = \Phi_{H_i}$.
The three examples discussed in this section illustrate that simultaneous convergence of any two of the three sequences does not imply the convergence of the third.
This demonstrates the necessity of all the hypotheses of Definition~\ref{dfn:topo-contact-dynamics} and of the uniqueness theorems in the previous section.
The first two examples are constructed locally on Euclidean space $\R^{2 n - 1}$ with its standard contact structure and standard contact form, and apply to any contact manifold of arbitrary dimension by Darboux's theorem.
The third example is constructed via contact Hamiltonians that depend only on time, and likewise can be constructed on any contact manifold.

In the first and most subtle example, the contact isotopies and their inverses uniformly converge to the identity, and the contact Hamiltonians generating these isotopies converge to the zero function, whereas the associated conformal factors diverge.

\begin{exa}[Divergence of conformal factors] \label{exa:divergence-factors}
Consider the standard contact form $\alpha = dz - \sum y_i \, dx_i$ on $\R^{2 n - 1}$.
The Reeb vector field is $R = \partial / \partial z$, and the contact vector field of a smooth contact Hamiltonian $H \colon [0,1] \times \R^{2 n - 1} \to \R$ is given by the identity
	\[ X_H^t = \sum_{i = 1}^{n - 1} \left( - \frac{\partial H_t}{\partial y_i} \right) \frac{\partial}{\partial x_i} + \sum_{i = 1}^{n - 1} \left( \frac{\partial H_t}{\partial x_i} + y_i \frac{\partial H_t}{\partial z} \right) \frac{\partial}{\partial y_i} + \left( H_t - \sum_{i = 1}^{n - 1} y_i \frac{\partial H_t}{\partial y_i} \right) \frac{\partial}{\partial z}. \]
For every positive integer $k > 1$, let $\eta_k \colon \R^{2 n - 2} \to [0,1]$ and $\rho_k \colon \R \to \R$ be smooth cut-off functions with the following properties.
Let $\epsilon_k$ be a sequence of positive real numbers converging to zero.
Then $\eta_k$ is a function of the variables $(x_1,y_1, \ldots , x_{n - 1}, y_{n - 1}) = (x,y) \in \R^{2 n - 2}$ that equals $1$ near the origin, and vanishes outside the ball of radius $\epsilon_k$ centered at the origin.
The function $\rho_k$ satisfies $\rho_k (0) = 0$, $\rho_k' (0) = 1$, and $\rho (z) = \pm \frac{\pi}{k^2 \ln k}$ for $| z | \ge \epsilon_k$.
By choosing $\epsilon_k$ appropriately, we can impose $| \rho'_k | \le 1$ is bounded independently of $k$.
Define a sequence of smooth contact Hamiltonians by
\begin{align} \label{eqn:divergence-factors}
	H_k (x,y,z) = \frac{\eta_k (x,y)}{k^2} \sin (k^2 \ln k \cdot \rho_k (z)).
\end{align}
As $k \to \infty$, the isotopies $\Phi_{H_k}$ and $\Phi_{H_k}^{-1}$ uniformly converge to the identity, because the Hamiltonians $H_k$ are supported in balls of shrinking radii $\sqrt{2} \epsilon_k$.
For every $k$, the Hamiltonian vector field $X_{H_k}$ vanishes at the origin, and thus the contact isotopies $\Phi_{H_k}$ and $\Phi_{H_k}^{-1}$ fix the origin $0 \in \R^{2 n - 1}$ at each time $t \in [0,1]$.
The conformal factor $h_k$ satisfies
	\[ h_{k}^t (0) = \int_0^t \left( \frac{\partial}{\partial z} H_k \right) \circ \phi_{H_k}^s (0) \, ds = \int_0^t \ln k \, ds = t \ln k, \]
and as a consequence, $\holine_k^t (0) = - t \ln k$.
In fact, $| h_k | = | \overline h_k | = \ln k$.
Finally, the two sequences $H_k$ and $\Hbar_k = - e^{- h_k} (H_k \circ \Phi_{H_k})$ of contact Hamiltonians uniformly converge to zero.
\end{exa}

In the next example, the sequences of contact Hamiltonians and conformal factors converge to zero uniformly, but the sequence of contact isotopies does not $C^0$-converge.

\begin{exa}[Divergence of contact isotopies]
Let $\epsilon_k > 0$ be a sequence of real numbers converging to zero, and $\rho$ be a smooth cut-off function, compactly supported near the origin in $\R^{2 n - 1}$, that equals $1$ on the line segment parametrized by $0 \leq x_1 \leq 1$.
Consider the sequence of autonomous Hamiltonians $H_k \colon \R^{2 n - 1} \to \R$ given by
	\[ H_k (x,y,z) = \rho (x,y,z) \cdot f_k (y_1), \]
where $f_k$ is a smooth function such that $f_k (0) = 0$, $f_k' (0) = - 1$, and $| f_k | \le \epsilon_k$.
The Hamiltonians $H_k$ and $\Hbar_k$ converge uniformly to the zero contact Hamiltonian, and the conformal factors $h_k$ and $\overline h_k$ uniformly converge to zero as well.
By construction, we have $\phi^t_{H_k} (0,\ldots,0,0,\ldots,0,0) = (t,0,\ldots,0,0,\ldots,0,0)$, and therefore $\dbar (\Phi_{H_k}, \id) \ge 1$, i.e.\ the distance to the identity is bounded from below by $1$.
By Corollary~\ref{cor:cds-unique}, the sequence $\Phi_k$ must diverge.
\end{exa}

In the final example, we reparametrize the isotopy generated by the Reeb vector field in such a way that the sequence of contact Hamiltonians does not converge, whereas the associated isotopies do $C^0$-converge.
The conformal factors are all identically zero.
This example is of a global nature, and applies to any contact manifold.

For a given contact Hamiltonian $H \colon [0,1] \times M \to \R$, generating the contact isotopy $\Phi_H = \{ \phi_H^t \}$, and any smooth function $\zeta \colon [0,1] \to [0,1]$, the reparametrized isotopy $\Phi_{H^\zeta} = \{ \phi_H^{\zeta (t)} \}$ is generated by the contact Hamiltonian $H^\zeta \colon [0,1] \times M \to \R$ defined by the formula
\begin{align} \label{eqn:ham-rep}
	H^\zeta (t,x) = \zeta' (t) H (\zeta(t), x).
\end{align}
We denote by $h^\zeta \colon [0,1] \times M \to \R$ the function given by $(\phi_{H^\zeta}^t)^* \alpha = e^{h^\zeta (t,\cdot)} \alpha$.
Clearly $h_\zeta (t,\cdot) = h (\zeta (t),\cdot)$, since $\phi_{H^\zeta}^t = \phi_H^{\zeta (t)}$.
This also follows from equation~(\ref{eqn:formula-conformal-factor}) by a simple change of variables in the integral.

\begin{exa}[Divergence of contact Hamiltonians]
To begin, consider the middle-thirds construction 
	\[ [0,1] = E_0 \supset E_1 \supset E_2 \supset \cdots\]
of the Cantor set $E = \bigcap E_k$ in the unit interval $[0,1]$.
We adhere to the presentation in \cite[7.16(b)]{rudin:rca87}.
At each stage of the construction, the set $E_k$ consists of $2^k$ disjoint intervals, and the lengths of each of these intervals equals $1/3^k$. 
For each $k = 1, 2, 3, \ldots$ define a step function $\tilde G_k$ by setting
	\[ \tilde G_k = \left( \frac{3}{2} \right)^k \cdot \chi_{E_k} \colon [0,1] \to \R ,\]
with an antiderivative $\tilde F_k \colon [0,1] \to [0,1]$ given by
	\[ \tilde F_k (t) = \int_0^t \tilde G_k(s) \, ds. \]
The sequence $\tilde F_k$ converges to the so called Cantor function $F \colon [0,1] \to [0,1]$ uniformly.
A rough lower bound for the $L^1$-difference between distinct functions $\tilde G_j$ and $\tilde G_k$ is given by $\| \tilde G_k - \tilde G_j \|_{L^1} \geq (1 - (2/3)^k ) \ge 5 / 9$ whenever $k > j$.

For each $k$, let $G_k$ be a smooth function suitably close to $\tilde G_k$ in the $L^1$-topology, so that $\| G_k - G_j \|_{L^1} \ge 1 / 2$ for distinct $j$ and $k$, and $\| G_k - \tilde G_k \|_{L^1} \to 0$.
Let $F_k$ denote as above the antiderivative with $F_k (0) = 0$.
The sequence of smooth functions $F_k \colon [0,1] \to [0,1]$ also uniformly converges to $F$, since
	\[ | F_k - F | \le | F_k - \tilde F_k | + | \tilde F_k - F | \le \| G_k - \tilde G_k \|_{L^1} + | \tilde F_k - F | \to 0. \]

Now consider the sequence $G_k$ as (space-independent) smooth contact Hamiltonians on $[0,1] \times M$ that generate smooth contact isotopies $\phi^t_{G_k}$.
The time-$t$ map satisfies $\phi^t_{G_k} = \phi_R^{F_k (t)}$, where $\{ \phi_R^t \}$ denotes the smooth contact isotopy generated by the Reeb vector field.

We make three observations.
The conformal factors $g_k$ are all identically zero since each function $G_k$ is basic.
Moreover, the sequence $\{ \phi^t_{G_k} \}$ of strictly contact isotopies $C^0$-converges to $\{ \phi_R^{F (t)} \}$, because $| F_k - F | \to 0$, as $k \to \infty$.
Finally, for every $j \neq k$, the contact norms satisfy $\| G_k - G_j \|_\alpha = \| G_k - G_j \|_{L^1} \ge 1 / 2$, and thus the contact Hamiltonians $G_k$ do not converge.
\end{exa}

It suffices in the last example to take a sequence of smooth functions $F_k$ on the unit interval that converges uniformly, but whose derivatives do not converge in $L^1$.

\section{Group properties} \label{sec:topo-group}
In order to simplify our subsequent arguments regarding Cauchy sequences with respect to the contact distance, we prove a useful lemma.

\begin{lem} \label{lem:cauchy}
Let $H_i$ and $h_i \colon [0,1] \times M \to \R$ be sequences of $L^\oneinfty$-functions and continuous functions, respectively, and $\Phi_i \colon [0,1] \times M \to M$ be a sequence of continuous isotopies of homeomorphisms.
Suppose that $d_M (\Phi_i,\Phi_j) \to 0$, $| h_i - h_j | \to 0$, and $\| H_i - H_j \|_\alpha \to 0$ as $i, j \to \infty$.
Then
	\[ \left\| e^{-h_i} (H_i \circ \Phi_i) - e^{-h_j} (H_j \circ \Phi_j) \right\|_\alpha \to 0, \]
as $i, j \to \infty$.
If $\Phi$ denotes the uniform limit of the sequence $\Phi_i$, $h$ the uniform limit of the sequence $h_i$, and $H$ the $L^\oneinfty$-limit of the sequence $H_i$, then the functions $e^{-h_i} (H_i \circ \Phi_i)$ converge to $e^{-h} (H \circ \Phi)$ in the metric induced by the norm $\| \cdot \|_\alpha$.
\end{lem}

\begin{proof}
In the special case $h_i = 0$ for all $i$, the statement of the lemma is verbatim the same as Proposition~2.3.9 in \cite{mueller:ghh08}, with the exception that the norm $\| \cdot \|$ there is defined for normalized functions, and thus is missing the average value term $c_\alpha$ that is present in equation~(\ref{eqn:contact-length}).
Arguing as in Lemma~\ref{lem:max-norm}, this does not affect the proof.
That shows the sequence $H_i \circ \Phi_i$ is Cauchy and converges to the function $H \circ \Phi$.
In particular, there exists a constant $c$ such that $\| H_i \circ \Phi_i \|_\alpha \le c$ for all $i$.
By choosing $c$ larger if necessary, $| h_i | \le c$ for all $i$.
To prove the lemma in full generality, we apply the triangle inequality and again Lemma~\ref{lem:max-norm}, and obtain
\begin{align*}
	{} & \| e^{-h_i} (H_i \circ \Phi_i) - e^{-h_j} (H_j \circ \Phi_j) \|_\alpha \\
	& \quad \le \| e^{-h_i} (H_i \circ \Phi_i) - e^{-h_j} (H_i \circ \Phi_i) \|_\alpha + \| e^{-h_j} (H_i \circ \Phi_i) - e^{-h_j} (H_j \circ \Phi_j) \|_\alpha \\
	& \quad \le 3 | e^{- h_i} - e^{- h_j} | \cdot \| H_i \circ \Phi_i \|_\alpha + 3 | e^{- h_j} | \cdot \| H_i \circ \Phi_i - H_j \circ \Phi_j \|_\alpha \\
	& \quad \le 3 c \cdot | e^{- h_i} - e^{- h_j} | + 3 e^c \cdot \| H_i \circ \Phi_i - H_j \circ \Phi_j \|_\alpha
\end{align*}
which converges to zero as $i, j \to \infty$.
\end{proof}

\begin{proof}[Proof of Theorem~\ref{thm:topo-group}]
Suppose $(\Phi_H,H,h)$ and $(\Phi_F,F,f)$ are topological contact dynamical systems with respect to a contact form $\alpha$.
By definition, one can find sequences $(\Phi_{H_i},H_i,h_i) \to (\Phi_H,H,h)$ and $(\Phi_{F_i},F_i,f_i) \to (\Phi_F,F,f)$ in the $d_\alpha$-contact metric, where $H_i$ and $F_i \colon [0,1] \times M \to \R$ are smooth contact Hamiltonians, $\Phi_{H_i}$ and $\Phi_{F_i}$ the corresponding contact isotopies, and $h_i$ and $f_i \colon [0,1] \times M \to \R$ the smooth functions defined by $(\phi_{H_i}^t)^* \alpha = e^{h_i (t,\cdot)} \alpha$ and $(\phi_{F_i}^t)^* \alpha = e^{f_i (t,\cdot)} \alpha$.
Then $\Phi_{H_i}^{-1} \circ \Phi_{F_i} \to \Phi_H^{-1} \circ \Phi_F$ in the $C^0$-metric, and consequently 
	\[ \overline h_i \# f_i= -h_i \circ \Phi_{H_i}^{-1} \circ \Phi_{F_i} + f_i \to -h \circ \Phi_H^{-1} \circ \Phi_F + f \]
uniformly over $(t,x) \in [0,1] \times M$, where we have used Lemma~\ref{lem:conformal-calculus}.
Moreover,
\begin{align*}
	{} & \| \Hbar_i \# F_i - \Hbar_j \# F_j \|_\alpha \\
	& \quad = \| e^{-h_i} ((F_i - H_i) \circ \Phi_{H_i}) - e^{-h_j} ((F_j - H_j) \circ \Phi_{H_j}) \|_\alpha \\
	& \quad \le \| e^{-h_i} (F_i \circ \Phi_{H_i}) - e^{-h_j} (F_j \circ \Phi_{H_j}) \|_\alpha + \| e^{-h_j} (H_j \circ \Phi_{H_j}) - e^{-h_i} (H_i \circ \Phi_{H_i}) \|_\alpha.
\end{align*}
By Lemma~\ref{lem:cauchy}, the expression in the last line converges to zero as $i, j \to \infty$.
Thus $(\Phi_{H_i}^{-1} \circ \Phi_{F_i},\Hbar_i \# F_i,\overline h_i \#f_i)$ is Cauchy in the contact metric, and its limit is the topological contact isotopy $(\Phi_H^{-1} \circ \Phi_F,\Hbar \# F,\holine \# f)$.
This does not depend on the choices of Cauchy sequences converging to $(\Phi_H,H,h)$ and $(\Phi_F,F,f)$.
In particular, $(\id,0,0)$ is the identity in $\TC (M,\alpha)$, and we have a well-defined composition and inverse.
Associativity of composition in the space $\TC (M,\alpha)$ is easily verified.
Thus we have shown that $\TC (M,\alpha)$ forms a group, and $\C (M,\alpha)$ is a subgroup.
Verifying that composition and inverse are continuous, or equivalently, that the map
	\[ ((\Phi_H, H, h),(\Phi_F, F, f)) \mapsto (\Phi_H^{-1} \circ \Phi_F, \Hbar \# F, \holine \# f) \]
is continuous, is a similar application of Lemma~\ref{lem:cauchy}.
\end{proof}

\begin{proof}[Proof of Theorem~\ref{thm:transformation-law}]
By definition, there exists a sequence of smooth contact dynamical systems $(\Phi_{H_i}, H_i, h_i)$ that converges to the topological contact dynamical system $(\Phi, H, h)$ in the contact metric, and a sequence of contact diffeomorphisms $\phi_i$ that $C^0$-converges to the homeomorphism $\phi$, and such that the sequence of smooth functions $g_i$ defined by $\phi_i^* \alpha = e^{g_i} \alpha$ converges uniformly to the continuous function $g$.
Recall that $\phi_i^{-1} \circ \Phi_{H_i} \circ \phi_i$ is generated by the contact Hamiltonian $e^{- g_i} (H_i \circ \phi_i)$, and
	\[ (\phi_i^{-1} \circ \phi_{H_i}^t \circ \phi_i)^* \alpha = e^{h_i \circ \phi_i + g_i - g_i \circ (\phi_i^{-1} \circ \Phi_{H_i} \circ \phi_i)} \alpha. \]
It therefore suffices to prove that
	\[ (\phi_i^{-1} \circ \Phi_{H_i} \circ \phi_i, e^{-g_i} (H_i \circ \phi_i), h_i \circ \phi_i + g_i - g_i \circ (\phi_i^{-1} \circ \Phi_{H_i} \circ \phi_i)) \]
converges in the contact metric to the topological contact dynamical system (\ref{eqn:transformation-law}).
The $C^0$-convergence of the isotopies and conformal factors is immediate.
On the other hand, by Lemma~\ref{lem:cauchy}
	\[ \| e^{- g_i} (H_i \circ \phi_i) - e^{- g_j} (H_j \circ \phi_j) \|_\alpha \to 0, \]
as $i, j \to \infty$, and the limit equals $e^{- g} (H \circ \phi)$.
\end{proof}

\section{A bi-invariant metric on the group of strictly contact diffeomorphisms} \label{sec:bi-invariant-metric}
The results of this section concern smooth strictly contact isotopies and their time-one maps.
Recall that the smooth contact Hamiltonian of a smooth strictly contact isotopy is invariant under the Reeb flow, and such a Hamiltonian is called basic.

The discovery of the Hofer metric prompted Banyaga and Donato to search for other classical diffeomorphism groups supporting a bi-invariant metric.
They studied the prequantization space of an integral symplectic manifold, and showed that under a certain topological assumption, the identity component $\Diff_0(M,\alpha)$ of the group of strictly contact diffeomorphisms indeed supports such a metric \cite{banyaga:lci06}.
A prequantization space consists of a contact manifold $M$, supporting a \emph{regular} contact form $\alpha$, whose Reeb flow induces a free $S^1$-action, and the quotient is an integral symplectic manifold $(B,\omega)$.
Their construction utilizes the non-degeneracy of the Hofer metric on the base $B$.
We extend their result to all contact manifolds, with no restrictions on the contact form $\alpha$.
The proof follows from Theorem~\ref{thm:energy-capacity-ineq}.

We first recall Banyaga and Donato's construction.
For a smooth strictly contact isotopy $\Phi_H$ generated by a smooth basic contact Hamiltonian $H$, its `length' is defined by
\begin{align} \label{eqn:banyaga-donato-length}
	\ell_{\rm BD}^\alpha (\Phi_H) = \int_0^1 \osc (H_t) dt + \left| \int_0^1 c_\alpha (H_t) dt \right|,
\end{align}
where again $\osc$ is the oscillation of a function on $M$, and $c_\alpha$ its average value (\ref{eqn:average-value}) with respect to the canonical volume form $\nu_\alpha = \alpha \wedge (d\alpha)^{n - 1}$.
A standard argument (see e.g.\ \cite{banyaga:scd97}) shows that equation~(\ref{eqn:average-value}) defines a surjective homomorphism on the universal covering space of $\Diff_0(M,\alpha)$,
	\[ c_\alpha \colon \widetilde{\Diff_0} (M,\alpha) \to \R. \]
The placement of the absolute value makes equation~(\ref{eqn:banyaga-donato-length}) differ from equation~(\ref{eqn:contact-length}), and $\ell_{\rm BD}^\alpha (\Phi_H) \leq \ell_\alpha (\Phi_H) = \| H \|_\alpha$.
We prefer the latter, because the homomorphism $c_\alpha$ vanishes on every loop $\Phi_H$ of strictly contact diffeomorphisms that is generated by a smooth basic Hamiltonian of the form $H (t,x) = f (t)$, with $\int_0^1 f (t) \, dt = 0$.
These are the only such isotopies.
In short, in the Reeb direction, the Banyaga--Donato length measures the net displacement after time $1$.

The \emph{contact energy} of a strictly contact diffeomorphism $\phi \in \Diff_0 (M,\alpha)$ is
	\[ E (\phi) = \inf_{H \mapsto \phi} \ell_{\rm BD}^\alpha (\Phi),\]
where the infimum is taken over all smooth basic contact Hamiltonians that generate the time-one map $\phi$.
The group structure on basic Hamiltonians and the transformation law imply the symmetry, triangle inequality, and invariance properties of the contact energy.

\begin{lem}{\cite[Lemma~1]{banyaga:lci06}}
Let $\phi$, $\psi \in \Diff_0 (M,\alpha)$, and $\theta \in \Diff (M,\alpha)$ be strictly contact diffeomorphisms.
The contact energy satisfies
\begin{align*}
	& E (\phi^{-1}) = E(\phi) & (\text{symmetry}), \\
	& E (\phi \circ \psi) \leq E(\phi) + E(\psi) & (\text{triangle \ inequality),} \\
	& E (\theta^{-1} \circ \phi \circ \theta) = E (\phi) & (\text{invariance}).
\end{align*}
\end{lem}

The proof that the contact energy of $\phi$ vanishes if and only if it is the identity is more difficult.
For a certain class of regular contact manifolds however, Banyaga and Donato demonstrate the following theorem.

\begin{thm}[Non-degeneracy of Banyaga--Donato metric]{\cite[Theorem~1]{banyaga:lci06}} \label{thm:banyaga-donato}
Suppose that $(M,\alpha)$ is a closed and connected regular contact manifold satisfying
\begin{align}\label{eqn:condition-c}
	{\rm Image} \left( c_\alpha \colon \pi_1(\Diff (M,\alpha)) \to \R \right) = \Z.
\end{align}
Then $E (\phi) = 0$ if and only if $\phi = \id$.
\end{thm}

By Theorem~\ref{thm:banyaga-donato}, the map $\Diff_0 (M,\alpha) \times \Diff_0(M,\alpha) \to \R$ given by $(\phi,\psi) \mapsto E (\phi^{-1} \circ \psi)$ defines a bi-invariant metric on $\Diff_0 (M,\alpha)$, provided $(M, \alpha)$ is a closed regular contact manifold that satisfies condition~(\ref{eqn:condition-c}).

We define an a priori different contact energy of a diffeomorphism $\phi$ in $\Diff_0(M,\alpha)$, by minimizing the contact length of equation~(\ref{eqn:contact-length}) over all strictly contact isotopies $\Phi_H$ whose time-one map equals $\phi$,
	\[ \underline{E} (\phi)  = \inf_{H \mapsto \phi} \| H \|_\alpha, \]
and prove a surprising fact.

\begin{lem}
Let $(M,\xi)$ be a contact manifold with a contact form $\alpha$.
Every strictly contact diffeomorphism $\phi \in \Diff_0 (M,\alpha)$ satisfies the identity $E (\phi) = \underline{E} (\phi)$.
\end{lem}

\begin{proof}
Given a smooth basic contact Hamiltonian $H \colon [0,1] \times M \to \R$, let $c_t = c_\alpha (H_t)$ be the average value of $H$ at time $t$, and $c = \int_0^1 c_t \, dt$ be the time-average of these averages.
Write $F_t = c - c_t$.
We claim that $\| H \#F \|_\alpha = \ell_{\rm BD}^\alpha (\Phi_H)$.
Indeed, denote by $\{ \phi_R^t \}$ the Reeb flow.
The smooth basic contact Hamiltonian $F$ generates a loop $\{ \phi_F^t \}$ of strictly contact diffeomorphisms, which is a reparametrization of the Reeb flow
	\[ \phi_F^t = \phi_R^{\int_0^t (c - c_s) ds}. \]
Because $F_t$ is independent of $x \in M$, $(H \# F)_t = H_t + F_t$, and $\osc (H_t + F_t) = \osc (H_t)$.
Furthermore
\begin{align} \label{eqn:avg-trick}
 	\int_0^1 \left| c_\alpha ((H\#F)_t) \right| \, dt = \int_0^1 \left| \frac{1}{\int_M \nu_\alpha} \int_M (H_t + F_t) \nu_\alpha \right| dt = | c |,
\end{align}
proving the claim.
Since $\phi_F^1 = \phi_R^0 = \id$, we have $\phi_{H \# F}^1 = \phi_H^1 \circ \phi_F^1 = \phi_H^1$.
Thus
	\[ \inf_{H \mapsto \phi} \ell_{\rm BD}^\alpha (\Phi_H) \geq \inf_{H \mapsto \phi} \| H \|_\alpha. \]
The reverse inequality is obvious.
\end{proof}

A simpler and also more general proof of non-degeneracy follows from our contact energy-capacity inequality in Theorem~\ref{thm:energy-capacity-ineq}.
There are no restrictions on the topology of $M$ or on the contact form $\alpha$.

\begin{thm}[Non-degeneracy of Banyaga--Donato metric] \label{thm:nondegen}
Let $\phi$ be a strictly contact diffeomorphism in $\Diff_0 (M,\alpha)$.
Then $E (\phi ) = 0$ if and only if $\phi = \id$.
\end{thm}

This theorem proves that the Banyaga--Donato pseudo-metric is non-degenerate for every $(M,\alpha)$, and completes the proof of Theorem~\ref{thm:non-degeneracy}.
Similarly, one defines a bi-invariant metric on any component of the group $\Diff (M,\alpha)$.
One can also set the distance of two strictly contact diffeomorphisms belonging to different components of $\Diff (M,\alpha)$ equal to $+ \infty$, and obtain a bi-invariant ``distance function'' on the whole group $\Diff (M,\alpha)$.

\begin{proof}
Suppose $\phi \in \Diff_0 (M,\alpha)$ and $E (\phi) = 0$.
Then there exists a sequence of strictly contact isotopies $\Phi_{H_i}$, such that
\begin{enumerate}
	\item [(i)] $\phi_{H_i}^1 = \phi$ for all $i$, and
	\item [(ii)] $\| H_i \|_\alpha \to 0$ as $i \to \infty$.
\end{enumerate}
The conformal factors $h_i$ are all identically zero, hence if $\phi \neq \id$, then by Theorem~\ref{thm:energy-capacity-ineq}, $\| H_i \|_\alpha \geq C > 0$, contradicting hypothesis (ii).
\end{proof}

In general, two Hamiltonian or strictly contact isotopies $\Phi_F$ and $\Phi_H$ satisfy
\begin{align} \label{eqn:length-preservation}
	\| \Hbar \# F\| = \| F - H \|,
\end{align}
where $\| \cdot \|$ refers to either the contact length in equation~(\ref{eqn:contact-length}) or the Hofer length in equation~(\ref{eqn:hofer-length}).
Equation~(\ref{eqn:length-preservation}) follows from the fact that the isotopies generated by $H$ and $F$ preserve the contact or symplectic form, respectively.
In either case, the composed isotopy $\Phi_H^{-1} \circ \Phi_F$ is generated by the function $\Hbar \# F = (F - H) \circ \Phi_H$.
On the other hand, two contact isotopies $\Phi_H$ and $\Phi_F$ with non-trivial conformal factors do not satisfy equation~(\ref{eqn:length-preservation}).
By Lemma~\ref{lem:ham-gp-str}, the composition $\Phi_H^{-1} \circ \Phi_F$ is generated by the contact Hamiltonian
	\[ \Hbar \# F = e^{-h} \left( (F - H) \circ \Phi_H \right). \]
We show by example that neither the function $\delta (\Phi_H, \Phi_F) = \ell (\Phi_H^{-1} \circ \Phi_F)$, nor its symmetrization $\frac{1}{2} (\delta (\Phi_H, \Phi_F) + \delta (\Phi_F, \Phi_H))$, satisfies the triangle inequality.
Recall from Example~\ref{exa:divergence-factors} the sequence of smooth contact Hamiltonians $H_k$, defined by equation~(\ref{eqn:divergence-factors}), whose conformal factors satisfy $| h_k | = \ln k$, and let $F = 1$.
Evaluating the functions
	\[ (H_k \# F)_t = H_k + e^{h_k^t \circ (\phi_{H_k}^t)^{-1}} \]
at the origin at time $1$ gives
	\[ \| H_k \# F \|_\alpha > k > \frac{3}{k^2} + 1 > \| H_k \|_\alpha + \| F \|_\alpha. \]
Hence $\delta (\Phi_{H_k}^{-1}, \Phi_F) > \delta (\Phi_{H_k}^{-1}, \id) + \delta (\id, \Phi_F)$.
Since $\| \Hbar_k \|_\alpha < 3 / k$ and $\Fbar = - 1$, the same conclusion holds for the symmetrization.
Adding the maximum norms of the conformal factors also does not prevent failure of the triangle inequality.
The example can be constructed any contact manifold $(M,\xi)$ with any contact form $\alpha$ by Darboux's theorem.

\section{Topological automorphisms and contact rigidity} \label{sec:automorphisms}
We now prove that the topological conformal factor of a topological automorphism $\phi$ is uniquely determined by the homeomorphism itself and by the contact form $\alpha$.

\begin{proof}[Proof of Proposition~\ref{thm:unique-topo-conformal-factor}]
By Lemma~\ref{lem:conformal-calculus}, $(\phi_i^{-1} \circ \psi_i)^* \alpha = e^{g_i - h_i \circ \phi_i^{-1} \circ \psi_i} \alpha$, and by our hypotheses the sequence of conformal factors $g_i - h_i \circ \phi_i^{-1} \circ \psi_i$ converges uniformly to the continuous function $g - h$.
Denote by
\begin{align} \label{eqn:normalized-vol}
	\mu_\alpha = \frac{1}{\int_M \alpha \wedge (d\alpha)^{n - 1}} \cdot \alpha \wedge (d\alpha)^{n - 1} = \frac{1}{\int_M \nu_\alpha} \cdot \nu_\alpha
\end{align}
the normalized canonical volume form induced by $\alpha$, and by $\mubar_\alpha$ the good measure on $M^{2 n - 1}$ in the terminology of \cite[Section 1]{fathi:sgh80}, obtained from integration of the volume form $\mu_\alpha$.
If $U \subseteq M$ is an open subset, then the sequence of real numbers
	\[ \int_U (\psi_i^{-1} \circ \phi_i)_* \mubar_\alpha = \int_U (\phi_i^{-1} \circ \psi_i)^* \mu_\alpha = \int_U e^{n (g_i - h_i \circ \phi_i^{-1} \circ \psi_i)} \mu_\alpha \rightarrow \int_U e^{n (g - h)} \mu_\alpha \]
as $i \to \infty$.
On the other hand, since the sequence of diffeomorphisms $\phi_i^{-1} \circ \psi_i$ $C^0$-converges to the identity, the induced measures $(\psi_i^{-1} \circ \phi_i)_* \mubar_\alpha$ converge to $\mubar_\alpha$ in the metric that induces the weak topology on the space of (good) measures on $M$ \cite[Proposition 1.5]{fathi:sgh80}.
Evaluating a measure at the set $U$ is lower semicontinuous by \cite{denker:etc76} or \cite[Proposition 1.2]{fathi:sgh80}, and therefore
\begin{align} \label{eqn:semicontinuous}
	\int_U e^{n (g - h)} \mu_\alpha \ge \int_U \mu_\alpha
\end{align}
for every open subset $U \subseteq M$.
This implies $e^{n (g - h)} \ge 1$.
Indeed, suppose $e^{n (g - h)} (x_0) < 1$ at some point $x_0 \in M$, then there exists an open neighborhood $U$ of $x_0$ such that $e^{n (g - h)} (x) < 1$ for all $x \in U$, a contradiction to equation~(\ref{eqn:semicontinuous}).
That proves $n (g - h) \ge 0$, or $g \ge h$.
Reversing the roles of $h$ and $g$ yields $h \ge g$, and hence the proof.
\end{proof}

The automorphism group $\Sympeo(W,\omega)$ of a symplectic manifold $(W,\omega)$ satisfies the Gromov--Eliashberg $C^0$-rigidity
	\[ \Sympeo (W,\omega) \cap \Diff (W) = \Symp (W,\omega). \]
The analogous result for the topological automorphism group $\Aut(M,\xi)$ is Theorem~\ref{thm:contact-rigidity}.
The hypotheses on convergence and smoothness in these theorems can be stated in many equivalent ways.
Recall from Section~\ref{sec:contact-geometry} that a $C^0$-Cauchy sequence of homeomorphisms $\phi_i$ always $C^0$-converges to another homeomorphism $\phi$.
On the other hand, if the sequence $\phi_i$ is only assumed to be uniformly Cauchy, then the limit $\phi$ is a continuous map, but in general not a homeomorphism.
In fact, the uniform metric $d_M$ is never complete on any manifold $M$.
However, as already pointed out in Section~\ref{sec:contact-geometry}, if in addition the limit $\phi$ is assumed to be a homeomorphism, then the sequence $\phi_i$ also $C^0$-converges to $\phi$.
Moreover, if each homeomorphism $\phi_i$ is volume-preserving (for example, if $\phi_i$ is a symplectic diffeomorphism), and the limit homeomorphism $\phi$ is assumed to be smooth, then it is in fact a volume-preserving diffeomorphism.
Indeed, the limit $\phi$ preserves volume, so if its derivative at a point exists, it has determinant $1$, and the claim follows from the inverse function theorem.
When the volume-preserving assumption is dropped, a smooth homeomorphism need not be a diffeomorphism, as the classical example $x \mapsto x^3$ on the real line shows.
However, a similar argument applies in the contact case, and Theorem~\ref{thm:contact-rigidity} turns out to be equivalent to
	\[ \Aut (M,\xi) \cap \Diff (M) = \Diff (M,\xi). \]

\begin{proof}[Proof of Theorem~\ref{thm:contact-rigidity}]
In local Darboux coordinates around a point $x \in M$, $\phi^* \mu_\alpha = \det d\phi (x) \cdot \mu_\alpha$, and by definition $\phi (x + y) - \phi (x) = d\phi (x) (y) + o (| y |)$.
Thus
\begin{align} \label{eqn:det-bound}
	\frac{(\phi^{-1})_* \mubar_\alpha (B_\epsilon)}{\mubar_\alpha (B_\epsilon)} = \frac{\mubar_\alpha (\phi (B_\epsilon))}{\mubar_\alpha (B_\epsilon)}\rightarrow \det d\phi (x)
\end{align}
as $\epsilon \to 0$, where $B_\epsilon$ is the closed ball of radius $\epsilon$ centered at $x$, and $\mubar_\alpha$ is the (signed) measure obtained by integration of the volume form $\mu_\alpha$.
The first step is to prove that $\det d\phi (x) > 0$ for all $x \in M$.
Then by the inverse function theorem, $\phi$ is an orientation-preserving diffeomorphism, and we can write $\phi^* \mu_\alpha = e^{n g} \mu_\alpha$ for a smooth function $g \colon M \to \R$.
Our argument does not depend on the choice of auxiliary Riemannian metric, and the atlas on $M$ can be chosen to be contact (in fact, strictly contact), and thus to preserve the orientation induced by the given volume form (in fact, the volume form itself).
Arguing by contradiction, suppose $\det d\phi (x) \le 0$ at the point $x \in M$.
By equation~(\ref{eqn:det-bound}), we can choose $\epsilon > 0$ small enough so that
	\[ (\phi^{-1})_* \mubar_\alpha (B_\epsilon) < \frac{1}{4} e^{- n | h |} \cdot \mubar_\alpha (B_\epsilon). \]
We have
	\[ e^{- n | h |} \cdot \mubar_\alpha (B_\epsilon) = e^{- n | h |} \cdot \int_{B_\epsilon} \mu_\alpha \le \int_{B_\epsilon} e^{n h} \mu_\alpha, \]
and since the functions $h_i$ converge to the continuous function $h$ uniformly,
	\[ \frac{1}{2} e^{- n | h |} \cdot \mubar_\alpha (B_\epsilon) \le \int_{B_\epsilon} e^{n h_i} \mu_\alpha = \int_{B_\epsilon} \phi_i^* \mu_\alpha = \int_{B_\epsilon} (\phi_i^{-1})_* \mubar_\alpha = (\phi_i^{-1})_* \mubar_\alpha (B_\epsilon) \]
for $i$ sufficiently large.
On the other hand, the sequence $\phi_i$ $C^0$-converges to the homeomorphism $\phi$, so the induced measures $(\phi_i^{-1})_* \mubar_\alpha \to (\phi^{-1})_* \mubar_\alpha$.
But evaluation at the compact set $B_\epsilon$ is upper semi-continuous, and we arrive at a contradiction.

Let $U \subseteq M$ be an open subset.
The diffeomorphisms $\phi_i^{-1} \circ \phi \colon M \to M$ $C^0$-converge to the identity, and thus the uniform convergence of the functions $h_i \to h$ implies
	\[ \int_U (\phi^{-1} \circ \phi_i)_* \overline \mu_\alpha = \int_U (\phi^{-1}_i \circ \phi)^* \mu_\alpha = \int_U e^{n (g - h_i \circ \phi_i^{-1} \circ \phi)} \mu_\alpha \rightarrow \int_U e^{n (g - h)} \mu_\alpha \]
as $i \to \infty$.
On the other hand, the measures $(\phi^{-1} \circ \phi_i)_* \overline \mu_\alpha$ converge in the weak topology to $\overline \mu_\alpha$, and since evaluating the measures on $U$ is lower semi-continuous, we have $g \ge h$.
Repeating the same argument with the sequence of inverses $\phi_i^{-1} \circ \phi$ proves $h \ge g$, and therefore $h = g$.
In particular $h$ is a smooth function.

Consider the symplectic diffeomorphisms $\widehat \phi_i (x,\theta) = (\phi_i (x), \theta - h_i (x))$ of the symplectization $(W,\omega)$ of $(M,\alpha)$.
Because of the choice of split Riemannian metric $g_W = \pi_1^*g_M + d\theta \otimes d\theta$, the sequence $\widehat \phi_i$ $C^0$-converges to the diffeomorphism $\widehat \phi$ given by $\widehat \phi (x, \theta) = (\phi(x), \theta - h (x))$.
Then by Gromov--Eliashberg $C^0$-rigidity, $\widehat \phi$ is a symplectic diffeomorphism, and by equation~(\ref{eqn:lifted-diffeo}), that is equivalent to $\phi$ being a contact diffeomorphism with $\phi^*\alpha = e^h \alpha$.
\end{proof}

See \cite{ms:gac13} for a rigidity result for contact diffeomorphisms that does not involve the conformal factors.
We remark however that if the uniform convergence of the conformal factors is dropped from the hypotheses of Theorem~\ref{thm:contact-rigidity}, then nothing can be said about the conformal factor of the limit diffeomorphism $\phi$ and the contact Hamiltonian of a contact isotopy conjugated by $\phi$.
See \cite{ms:tcd2, ms:gac13} for details.

In the article \cite{banyaga:ugh11}, the automorphism group of the contact form $\alpha$ was defined to be the $C^0$-closure $\overline{\Diff} (M,\alpha)$ of the group of strictly contact diffeomorphisms of $(M,\alpha)$ in the group $\Homeo (M)$ of homeomorphisms of $M$.
A homeomorphism $\phi$ belongs to this group if and only if there exists a sequence of strictly contact diffeomorphisms that uniformly converges to $\phi$.
More generally, in this paper we define the topological automorphism group $\Aut (M,\alpha)$ of the contact form $\alpha$ as the subgroup of $\Aut (M,\xi)$ consisting of those homeomorphisms $\phi$ with topological conformal factor equal to zero.

\begin{dfn}[Topological automorphism of the contact form]
A homeomorphism belongs to the subgroup $\Aut (M,\alpha) \subset \Aut (M,\xi)$, called the \emph{topological automorphism group} of the \emph{contact form} $\alpha$, if its unique topological conformal factor $h_\phi$ vanishes,
	\[ \Aut (M,\alpha) = \{ \phi \in \Aut (M,\xi) \mid h_\phi = 0 \} \subset \Aut (M,\xi). \]
In other words, $\phi \in \Aut (M,\alpha)$ if there exists a sequence of contact diffeomorphisms $\phi_i$ with $\phi_i^* \alpha = e^{h_i} \alpha$, such that the sequence $\phi_i$ $C^0$-converges to the homeomorphism $\phi$, and the sequence of conformal factors $h_i$ converges uniformly to zero.
\end{dfn}

As in the case of diffeomorphisms,
	\[ \Aut (M,\alpha) = \Aut (M,\xi) \cap \Homeo (M,\mubar_\alpha). \]
By Theorem~\ref{thm:transformation-law}, conjugation by $\Aut (M,\alpha)$ preserves the group of topological strictly contact dynamical systems of $(M,\alpha)$, and in particular
	\[ \Homeo (M,\alpha) \trianglelefteq \Aut (M,\alpha), \]
where $\Homeo (M,\alpha)$ denotes the group of time-one maps of topological strictly contact isotopies.
The last statement appears in \cite{banyaga:ugh11} with the group $\Aut (M,\alpha)$ replaced by $\overline{\Diff} (M,\alpha)$.
Clearly $\overline{\Diff} (M,\alpha) \subseteq \Aut (M,\alpha)$.
In \cite{banyaga:ugh11}, it is shown that if the contact form $\alpha$ is regular, then $\Diff (M,\alpha) \subsetneq \Homeo (M,\alpha) \subseteq \overline{\Diff} (M,\alpha)$. 

The next two corollaries are consequences of Theorem~\ref{thm:contact-rigidity}.

\begin{cor} \label{cor:contact-rigidity}
Let $M$ be a contact manifold with a contact form $\alpha$.
Then
	\[ \Aut (M,\alpha) \cap \Diff (M) = \Diff (M,\alpha). \]
\end{cor}

The same statement also holds with $\Aut (M,\alpha)$ replaced by $\overline{\Diff} (M,\alpha)$.

\begin{cor}[Strictly contact $C^0$-rigidity] \label{cor:strictly-contact-rigidity}
Let $\alpha$ be a contact form on a contact manifold $(M,\xi)$.
The group of strictly contact diffeomorphisms of $(M,\alpha)$ is $C^0$-closed in the group of diffeomorphisms of $M$.
In other words, suppose the sequence $\phi_i$ of strictly contact diffeomorphisms uniformly converges to a homeomorphism $\phi$ of $M$, and assume that $\phi$ is smooth.
Then $\phi$ is a diffeomorphism that preserves $\alpha$, i.e.\ $\phi^*\alpha = \alpha$.
\end{cor}

It is not difficult to see that the inclusions $\Diff (M,\xi) \subset \Diff (M)$ and $\Diff (M,\alpha) \subset \Diff (M,\mu_\alpha)$ are in fact proper, and by contact $C^0$-rigidity, this immediately implies $\Aut (M,\xi) \subsetneq \Homeo (M)$ and $\Aut (M,\alpha) \subsetneq \Homeo (M,\mubar_\alpha)$.
The same statements hold for the identity components.

As the given proofs clearly show, the rigidity results proved in this section are local statements, and thus local versions of these results hold.
In particular, they generalize to open manifolds.
We give the proof only for a local version of Proposition~\ref{thm:unique-topo-conformal-factor}.

\begin{pro}[Local uniqueness of topological conformal factor] \label{pro:local-unique-topo-conformal-factor}
Let $U \subseteq M$ be an open subset, and $\phi_i$ and $\psi_i \in \Diff (M)$ be two sequences with $\phi_i^* \alpha = e^{h_i} \alpha$ and $\psi_i^* \alpha = e^{g_i} \alpha$ on $U$.
Suppose that the sequences $\phi_i^{-1} \circ \psi_i$ and $\psi_i^{-1} \circ \phi_i$ converge to the identity uniformly on compact subsets of $U$, and moreover, the sequences $h_i$ and $g_i$ converge uniformly on compact subsets of $U$ to continuous functions $h$ and $g$, respectively.
Then $h = g$ on $U$.
\end{pro}

Note that we do not require the stronger assumption that the diffeomorphisms $\phi_i^{-1} \circ \psi_i$ and $\psi_i^{-1} \circ \phi_i$ of $M$ preserve the subset $U$, but only that for all $x \in U$, $\lim_i \phi_i^{-1} \circ \psi_i (x) = x = \lim_i \psi_i^{-1} \circ \phi_i (x)$, and that this convergence is uniform on compact subsets of $U$.

\begin{proof}
Let $x \in U$, and choose open neighborhoods $V$ and $V'$ of $x$ with compact closures, such that $\overline{V} \subset V' \subset \overline{V}' \subset U$.
By hypothesis, the diffeomorphisms $\phi_i^{-1} \circ \psi_i$ converge to the identity uniformly on $\overline{V}$, and in particular, $\phi_i^{-1} \circ \psi_i (\overline{V}) \subset V'$ for $i$ sufficiently large.
Thus since $h_i$ and $g_i$ converge uniformly on $\overline{V}'$, the conformal factor $g_i - h_i \circ \phi_i^{-1} \circ \psi_i$ converges to $g - h$ uniformly on $\overline{V}$.
Arguing as in the proof of Proposition~\ref{thm:unique-topo-conformal-factor}, we conclude $g - h \ge 0$ on $V$, and in particular $g (x) \ge h (x)$.
Since $x \in U$ was arbitrary, this proves $g \ge h$ on $U$.
Again reversing the roles of $g$ and $h$ yields $g \le h$, and hence the proof.
\end{proof}

Conversely, suppose the topological conformal factor $h$ of a topological automorphism $\phi$ is smooth.
This does not necessarily imply that $\phi$ is a (contact) diffeomorphism.
Indeed, when the contact form $\alpha$ is regular, there exist strictly contact homeomorphisms that are not smooth (or even $C^1$) \cite{banyaga:ugh11}.

\begin{que}[Smooth topological conformal factor] \label{que:smooth-conformal-factor}
Suppose the topological conformal factor $h$ of the topological automorphism $\phi$ is smooth.
What can be said about the properties of the homeomorphism $\phi$?
Does there exist a contact diffeomorphism $\psi$ with $\psi^* \alpha = e^h \alpha$?
\end{que}

This question will be discussed in the next section.
In \cite{banyaga:ecg00}, Banyaga defined an invariant $D_\alpha \colon \Diff (M,\xi) \longrightarrow C^\infty (M)$ of a contact diffeomorphism $\phi$, by assigning to it the conformal factor $\overline f = f_{\phi^{-1}}$ of its inverse.
The map $D_\alpha$ is a one-cocycle, and the cohomology class $[D_\alpha]$ in $H^1 (\Diff (M,\xi), C^\infty (M))$ it represents does not depend on the contact form $\alpha$.
We call $D_\alpha$ \emph{Banyaga's cohomological conformal contact invariant}.
This invariant obviously generalizes to automorphisms of $\xi$ by setting
	\[ \overline{D}_\alpha \colon \Aut (M,\xi) \longrightarrow C^0 (M), \quad \phi \mapsto \overline f, \]
where $\overline f = f_{\phi^{-1}}$ is the unique topological conformal factor associated to the topological automorphism $\phi^{-1} \in \Aut(M,\xi)$.

\begin{pro}
The function $\overline{D}_\alpha$ is a one-cocycle with values in $C^0(M)$, whose cohomology class $[\overline{D}_\alpha] \in H^1 (\Aut (M,\xi), C^0 (M))$ is independent of the contact form $\alpha$ defining $\xi$.
Moreover, if $\phi$ is smooth, then $\overline{D}_\alpha (\phi) = D_\alpha (\phi)$.
\end{pro}

\begin{proof}
The first statement follows at once from Theorem~\ref{thm:unique-topo-conformal-factor}, Proposition~\ref{pro:gp-str-auto}, and Proposition~\ref{pro:indep-auto-gp}.
The last part is a consequence of Theorem~\ref{thm:contact-rigidity}.
\end{proof}

Let $\sigma$ be the conformal class of a tensor field on a smooth manifold $M$.
The group $\Diff (M,\sigma)$ of diffeomorphisms preserving $\sigma$ is called \emph{inessential} \cite{banyaga:ecg00} if there exists a representative tensor $\tau \in \sigma$, such that $\Diff (M,\sigma) = \Diff (M,\tau)$.
If no such tensor in the conformal class $\sigma$ exists, the group is called \emph{essential}.
Banyaga proved that the group $\Diff (M,\sigma)$ is inessential if and only if $[D_\alpha] = 0$.
Along with some other classical structures, Banyaga showed that the group of contact diffeomorphisms is essential.
Using a local argument and a cohomological equation, it is shown in \cite{ms:pec12} that the group of contact diffeomorphisms of any contact manifold is in fact \emph{properly essential}, i.e.\
	\[ \bigcup_{\alpha} \Diff (M,\alpha) \varsubsetneq \Diff (M,\xi), \]
where the union is over all contact forms $\alpha$ with $\ker \alpha = \xi$.
The topological automorphism group $\Aut (M,\xi)$ exhibits similar behavior.

\begin{thm}[Proper essentiality]
The group of topological automorphisms of the contact structure $\xi$ is properly essential, i.e.\
\begin{align} \label{eqn:properly-essential}
	\bigcup_\alpha \Aut (M,\alpha) \varsubsetneq \Aut (M,\xi),
\end{align}
where the union is over all contact forms with $\ker \alpha = \xi$.
The cohomology class $[\overline{D}_\alpha] \in H^1(\Aut (M,\xi), C^0 (M))$ is non-vanishing.
\end{thm}

\begin{proof}
On every contact manifold $(M,\xi)$, there exists a contact diffeomorphism $\phi$ that does not preserve any contact form $\alpha$ with $\ker \alpha = \xi$ \cite{ms:pec12}.
Since $\Diff (M,\xi) \subset \Aut (M,\xi)$, the diffeomorphism $\phi$ belongs to $\Aut (M,\xi)$.
If $\phi \in \Aut (M,\alpha)$ for some contact form $\alpha$, then $\phi \in \Diff (M,\alpha)$ by Corollary~\ref{cor:contact-rigidity}, a contradiction.
That proves equation~(\ref{eqn:properly-essential}).
Suppose $[\overline{D}_\alpha] = 0$, then there exists a \emph{continuous} function $f$ such that $\overline{D}_\alpha (\phi) = f - f \circ \phi$ for all $\phi$.
But it is shown in \cite{ms:pec12} that for any $(M,\xi)$, $x \in M$, and arbitrary constant $k$, there exist neighborhoods $U \subset V$ of $x$, and a contact diffeomorphism $\phi$, compactly supported in $V$, with $D_\alpha (\phi) = k$ on $U$.
In fact, $x$ is a fixed point of $\phi$, and $D_\alpha (\phi) (x) = k$ for \emph{every} contact form $\alpha$.
Since $\overline{D}_\alpha (\phi) = D_\alpha (\phi)$, the final claim follows.
\end{proof}

Similarly, one shows that
	\[ \bigcup_\alpha \Homeo (M,\alpha) \subsetneq \Homeo (M,\xi). \]
The cohomological equation that establishes proper essentiality of the conformal group $\Diff (M,\xi)$ is the following.
Suppose $\phi \in \Diff (M,\xi)$, and $\alpha$ is a contact form with $\ker \alpha = \xi$.
Then $\phi^* \alpha = e^h \alpha$ for a smooth function $h$ on $M$.
Any other contact form on $(M,\xi)$ has the form $\alpha' = e^f \alpha$ for some smooth function $f$ on $M$, and the diffeomorphism $\phi$ preserves $\alpha'$ if and only if
\begin{align} \label{eqn:cohomological}
	h = f - f \circ \phi.
\end{align}
In other words, $\phi$ preserves some contact form on $(M,\xi)$, if and only if there exists a smooth solution to the \emph{cohomological equation}~(\ref{eqn:cohomological}) \cite{ms:pec12}.
As it turns out, it is often easier to find \emph{continuous} solutions to this cohomological equation, or obstructions to the existence of continuous solutions of equation~(\ref{eqn:cohomological}).
The following lemma complements the results in \cite{ms:pec12}.

\begin{lem}
Suppose $\phi \in \Diff (M,\xi)$ with $\phi^* \alpha = e^h \alpha$, and $f \in C^0 (M)$ is a continuous solution to the cohomological equation~(\ref{eqn:cohomological}).
Then for every $\epsilon > 0$, there exists a contact form $\alpha_\epsilon$ on $(M,\xi)$, such that $\phi^* \alpha_\epsilon = e^{h_\epsilon} \alpha$, and $| h_\epsilon | < \epsilon$.
\end{lem}

In other words, if there exists a smooth solution to the cohomological equation~(\ref{eqn:cohomological}), then for an appropriate choice of contact form on $(M,\xi)$, the conformal factor of $\phi$ vanishes.
If the solution to equation~(\ref{eqn:cohomological}) is merely continuous, then by choosing the contact form appropriately, the conformal factor of $\phi$ can be made arbitrarily small in the maximum norm.

\begin{proof}
Choose a sequence of smooth functions $f_i$ that converges uniformly to the continuous function $f$, and write $\alpha_i = e^{f_i} \alpha$.
The conformal factor of $\phi$ with respect to $\alpha_i$ equals $h + (f_i \circ \phi - f_i)$, and this sequence of functions converges uniformly to $h + (f \circ \phi - f) = 0$ as $i \to \infty$.
\end{proof}

\section{Topological conformal factors and topological Reeb flows} \label{sec:smooth-conformal-factors}
Recall that a necessary condition for the existence of a contact diffeomorphism $\psi$ such that $\psi^* \alpha = e^h \alpha$ is that the Reeb flows of $\alpha$ and $e^h \alpha$ are (smoothly) conjugate.
Thus on every closed contact manifold $(M,\xi)$, with any contact form $\alpha$, there exists a smooth function $h$ with the property that $\alpha$ and $e^h \alpha$ are not diffeomorphic \cite{ms:pec12}.
Denote by $\{ \phi_R^t \}$ the Reeb flow of the Reeb vector field $R$ corresponding to the contact form $\alpha$.
The following lemma is a special case of Theorem~\ref{thm:transformation-law}.

\begin{lem} \label{lem:conjugation}
If $\phi \in \Aut (M,\xi)$ is a topological automorphism with topological conformal factor $h$, then the isotopy $\{ \phi^{-1} \circ \phi_R^t \circ \phi \}$ is a topological contact isotopy with topological contact Hamiltonian $e^{- h}$ (and topological conformal factor $h - h \circ \phi^{-1} \circ \phi_R^t \circ \phi$).
\end{lem}

A partial answer to Question~\ref{que:smooth-conformal-factor} is the following proposition.

\begin{pro} \label{pro:reeb-flows}
If the topological conformal factor $h$ of a topological automorphism $\phi$ is smooth, then the Reeb flows of the contact forms $\alpha$ and $\alpha' = e^h \alpha$ are conjugated by the homeomorphism $\phi$.
\end{pro}

Recall however that two topologically conjugate smooth (strictly contact) vector fields are not necessarily $C^1$-conjugate \cite{ms:hvf12}.

\begin{proof}[Proof 1]
Let $R'$ be the Reeb vector field of $\alpha'$, and $\{ \phi_{R'}^t \}$ be its Reeb flow.
The smooth vector field $R'$ is contact with respect to the contact structure $\xi$, and its smooth contact Hamiltonian with respect to the contact form $\alpha$ is the function $\alpha (R') = e^{- h} \alpha' (R') = e^{- h}$.
By uniqueness of the topological contact isotopy associated to a given topological contact Hamiltonian (in this case, the function $H = e^{- h}$), the isotopies $\{ \phi^{-1} \circ \phi_R^t \circ \phi \}$ and $\{ \phi_{R'}^t \}$ coincide.
\end{proof}

\begin{lem}
If $\phi \in \Aut (M,\xi)$, and its topological conformal factor with respect to a contact form $\alpha$ vanishes, then $\phi$ commutes with the Reeb flow $\{ \phi_R^t \}$ of $\alpha$ at each time $t$.
\end{lem}

This result is proved in \cite{banyaga:ugh11} under the hypothesis that $\phi$ is the time-one map of a topological strictly contact isotopy.

\begin{proof}
The homeomorphism $\phi$ commutes with the time-$t$ map $\phi_R^t$, if and only if $\phi_R^t = \phi^{-1} \circ \phi_R^t \circ \phi$.
If the topological conformal factor of $\phi$ vanishes, both topological contact isotopies correspond to the constant topological contact Hamiltonian $H = 1$, and by uniqueness of the topological contact isotopy, they must coincide.
\end{proof}

\begin{lem}
The topological contact isotopy $\{ \phi^{-1} \circ \phi_R^t \circ \phi \}$ depends only on the topological conformal factor of the topological automorphism $\phi$.
That is, if $\psi \in \Aut (M,\xi)$ is another topological automorphism with the same topological conformal factor as $\phi$, then
	\[ \phi^{-1} \circ \phi_R^t \circ \phi = \psi^{-1} \circ \phi_R^t \circ \psi. \]
\end{lem}

\begin{proof}
Since by Proposition~\ref{pro:gp-str-auto} $\phi \circ \psi^{-1}$ has topological conformal factor zero,
	\[ \phi^{-1} \circ \phi_R^t \circ \phi = \psi^{-1} \circ ((\phi \circ \psi^{-1})^{-1} \circ \phi_R^t \circ (\phi \circ \psi^{-1})) \circ \psi = \psi^{-1} \circ \phi_R^t \circ \psi \]
by the previous lemma.
Alternatively, this follows directly from Lemma~\ref{lem:conjugation}, since both isotopies correspond to the same topological contact Hamiltonian.
\end{proof}

Thus we may define the topological Reeb flow of $e^h \alpha$ for any topological conformal factor $h$.

\begin{dfn}[Topological Reeb flow]
Given a topological conformal factor $h$, then the \emph{topological Reeb flow} of the continuous one-form $e^h \alpha$ is the topological contact isotopy $\{ \phi^{-1} \circ \phi_R^t \circ \phi \}$, where $R$ is again the Reeb vector field of the contact form $\alpha$, and $\phi$ is any topological automorphism of $\xi$ with topological conformal factor $h$.
\end{dfn}

By the preceding lemma, this definition does not depend on the particular choice of topological automorphism $\phi$ with topological conformal factor $h$.
Moreover, we have seen in Proposition~\ref{pro:reeb-flows} that if $h$ is smooth, then this coincides with the usual definition of the Reeb flow.
More generally, the definition does not depend on the choice of contact form either in the following sense.

\begin{lem}
Suppose $\phi \in \Aut (M,\xi)$ is a topological automorphism with topological conformal factor $h$ with respect to $\alpha$, $\alpha' = e^f \alpha$ is another contact form on $(M,\xi)$, and $\psi \in \Aut (M,\xi)$ is a topological automorphism with topological conformal factor $h - f$ with respect to $\alpha'$.
Then $\phi^{-1} \circ \phi_R^t \circ \phi = \psi^{-1} \circ \phi_{R'}^t \circ \psi$.
\end{lem}

\begin{proof}
The topological automorphism $\phi \circ \psi^{-1}$ has topological conformal factor $f$ with respect to the contact form $\alpha$.
Thus by Proposition~\ref{pro:reeb-flows},
	\[ \phi^{-1} \circ \phi_R^t \circ \phi = \psi^{-1} \circ ((\phi \circ \psi^{-1})^{-1} \circ \phi_R^t \circ (\phi \circ \psi^{-1})) \circ \psi = \psi^{-1} \circ \phi_{R'}^t \circ \psi. \qedhere \]
\end{proof}
An alternate proof of Proposition~\ref{pro:reeb-flows} follows from a parametrized version of a theorem due to E.~Opshtein \cite[Theorem~1]{opshtein:crc09} that we now explain.
Opshtein showed that if $S$ and $S'$ are smooth hypersurfaces of symplectic manifolds $(W,\omega)$ and $(W',\omega')$, respectively, then a symplectic homeomorphism $W' \to W$ which sends $S'$ to $S$ interchanges the characteristic foliations of $S'$ and $S$.
Proof 1 for Proposition~\ref{pro:reeb-flows} suggests there exists a proof of Opshtein's Theorem for \emph{parametrized} characteristic foliations using topological Hamiltonian dynamics.
This is indeed the case under some additional hypotheses.

Let $S$ be a compact and orientable hypersurface, and choose a compactly supported smooth function $H \colon W \to \R$, such that $1$ is a regular value of $H$, and $S \subseteq H^{-1} (1)$ is a component of the regular level set of $H$.
Such a function always exists.
The leaves of the characteristic foliation of $S$ are the unparametrized integral curves of the Hamiltonian flow generated by the function $H$, independent of the particular choice of function $H$ with the above properties.
Since $S$ is compact and regular, there exists an open and bounded neighborhood $U$ of $S$, which is filled with a family of compact and regular hypersurfaces $S_\lambda \subseteq H^{-1} (\lambda)$ parametrized by the energy, where $\lambda$ belongs to an open interval $I$ around $1$.
Moreover, each $S_\lambda$ is diffeomorphic to $S$, which corresponds to the parameter value $\lambda = 1$, and this defines a diffeomorphism $S \times I \to U$.
In these coordinates on $U$, the function $H$ is given by $H (x,\lambda) = \lambda$.
See for example Chapter~4 of \cite{hofer:sih94} for details.
The family $S_\lambda$, $\lambda \in I$ is called a \emph{one-parameter family of hypersurfaces modeled on $S$}.

\begin{thm} \label{thm:parametrized-opshtein}
Let $S$ and $S'$ be compact and orientable smooth hypersurfaces of symplectic manifolds $(W,\omega)$ and $(W',\omega')$, $S_\lambda$ and $S'_\lambda$ be one-parameter families of hypersurfaces modeled on $S$ and $S'$, and $H$ and $H'$ be smooth functions defining parametrizations of the characteristic foliations of $S_\lambda$ and $S'_\lambda$, respectively.
Then a symplectic homeomorphism $\phi \colon W' \to W$ that sends each $S'_\lambda$ to the corresponding $S_\lambda$, interchanges the parametrized characteristic foliations of $S_\lambda$ and $S'_\lambda$ for all $\lambda$.
\end{thm}

\begin{proof}
By the transformation law from \cite{mueller:ghh07}, the homeomorphism $\phi$ conjugates the topological Hamiltonian isotopies corresponding to the topological Hamiltonian functions $H$ and $H \circ \phi$.
By hypothesis, $H (\phi (x,\lambda)) = \lambda = H' (x,\lambda)$ on $U'$.
After multiplying $H$ with a smooth cut-off function $\rho (\lambda)$ on $I$ that equals $1$ on an open subinterval $J \subset I$, we may assume $H \circ \phi = H'$ on all of $W'$.
The uniqueness of the topological Hamiltonian isotopy from \cite{mueller:ghh07} implies the topological Hamiltonian isotopies corresponding to $H'$ and $H \circ \phi$ coincide.
Thus $\phi$ interchanges the parametrized characteristic foliations of $S'_\lambda$ and $S_\lambda$ for all $\lambda \in J$.
Since $J$ is arbitrary, this holds for all $\lambda \in I$.
\end{proof}

In fact, for the conclusion that the characteristic foliations of $S'$ and $S$ only are interchanged, instead of assuming $\phi$ sends each $S'_\lambda$ to $S_\lambda$, it suffices that the function $H \circ \phi$ is smooth in an open neighborhood of $S'$.

\begin{proof}[Proof 2 of Proposition~\ref{pro:reeb-flows}]
Denote by $(W,\omega)$ and $(W,\omega')$ the symplectizations of $(M,\alpha)$ and $(M,\alpha')$, respectively, and let $\phihat \colon (W,\omega) \to (W,\omega)$ be the lift of the topological automorphism $\phi$ to a symplectic homeomorphism $\phihat$ defined by $\phihat (x,\theta) = (\phi (x),\theta - h (x))$.
Since $h$ is smooth, the map $\phi_h \colon (W,\omega') \to (W,\omega)$ defined in Section~\ref{sec:symplectization} by $(x,\theta) \mapsto (x,\theta + h (x))$ is a symplectic diffeomorphism.
Then the composition $\phihat \circ \phi_h \colon (W,\omega') \to (W,\omega)$ is a symplectic homeomorphism, which sends each hypersurface of the form $M_\theta = M \times \{ \theta \} \subset M \times \R$ to itself. 
Let $\wH (x,\theta) = e^\theta$ be the lift of the constant contact Hamiltonian $H = 1$ generating the Reeb flows of $(M,\alpha)$ and $(M,\alpha')$.
By Theorem~\ref{thm:parametrized-opshtein}, the homeomorphism $\phihat \circ \phi_h$ interchanges the parametrized characteristic foliations of $M_0$ with respect to $\omega$ and $\omega'$.
Thus the homeomorphism $\phi \colon M \to M$ conjugates the Reeb flows of $\alpha$ and $\alpha'$.
\end{proof}

Opshtein's proof of his theorem uses the notion of a \emph{symplectic hammer}, which by definition is a symplectic isotopy of a symplectic manifold satisfying certain properties.
Since symplectic hammers are supported in Darboux balls, every smooth symplectic hammer is Hamiltonian, or a \emph{smooth Hamiltonian hammer}.
The correct generalization to continuous isotopies seems to be to topological Hamiltonian isotopies as defined in \cite{mueller:ghh07}, and satisfying the same properties as in Definition~1.1 in \cite{opshtein:crc09}.
We call this a \emph{topological Hamiltonian hammer}.
Moreover, a \emph{topological symplectic hammer} should be defined as the $C^0$-limit of smooth symplectic hammers.
Indeed, part of the proof of Opshtein's main theorem requires the approximation of a continuous symplectic hammer by a smooth symplectic hammer.
All the results in \cite{opshtein:crc09} hold with this notion of topological Hamiltonian or symplectic hammer replacing symplectic hammers.
If two points in the intersection $S \cap B$ of a symplectic hypersurface with a Darboux ball lie in the same characteristic, then there exists a smooth symplectic $\epsilon$-hammer between them for any small $\epsilon > 0$ \cite{opshtein:crc09}, and this hammer is of course a topological Hamiltonian (or symplectic) hammer.
Conversely, if for each small $\epsilon > 0$, a topological Hamiltonian (or symplectic) $\epsilon$-hammer between two given points in the intersection $S \cap B$ exists, they lie in the same characteristic.
The proof is verbatim the same as in \cite{opshtein:crc09}.

An ad hoc definition of a \emph{symplectic $C^0$-submanifold} is the image of a smooth symplectic submanifold by a symplectic homeomorphism.
By Opshtein's Theorem, one can define the \emph{topological characteristic foliation} of a $C^0$-symplectic hypersurface $S$ as the image of the characteristic foliation of any smooth symplectic hypersurface that is mapped to $S$ by a symplectic homeomorphism.
This is well-defined, and coincides with the usual definition of characteristic foliation if $S$ is smooth.
In fact, Opshtein defines the characteristic foliation in terms of symplectic hammers, which gives rise to an equivalent definition of topological characteristic foliations.
The unparametrized topological Reeb flow of $e^h \alpha$ then coincides with the topological characteristic foliation of the $C^0$-hypersurface $\phihat (M_0)$ of the symplectization $M \times \R$ of $(M,\alpha)$, provided $\phi$ is a topological automorphism of $(M,\xi)$ with topological conformal factor $h$ with respect to the contact form $\alpha$.

\section{The sequels: topological contact dynamics II and III} \label{sec:sequels}
\subsection{On the choice of contact metric}
Hofer \cite{hofer:tps90} originally defined the `length' of a Hamiltonian isotopy $\Phi_H$ of $\R^{2 n}$ with its standard symplectic structure, by the maximum oscillation over time of its unique compactly supported generating smooth Hamiltonian,
\begin{align}\label{eqn:hofer-inf}
	\| H \|_{\rm Hofer}^\infty = \max_{0 \le t \le 1} \left( \max_{x \in \R^{2 n}} H (t,x) - \min_{x \in \R^{2 n}} H (t,x) \right).
\end{align}
Polterovich subsequently adopted the $L^\oneinfty$-norm in equation~(\ref{eqn:hofer-length}), and showed that these two definitions descend to equal pseudo-norms on the group of Hamiltonian diffeomorphisms, i.e.\ if $\phi$ is a Hamiltonian diffeomorphism of a symplectic manifold $(W,\omega)$, then \cite{polterovich:ggs01}
	\[ \inf_{H \mapsto \phi} \| H \|_{\rm Hofer} = \inf_{H \mapsto \phi} \| H \|_{\rm Hofer}^\infty. \]
Non-degeneracy follows from the energy-capacity inequality~(\ref{eqn:energy-capacity-ineq}).
For smooth Hamiltonian isotopies $\Phi_H$ and $\Phi_F$, generated by normalized smooth time-dependent Hamiltonian functions $H$ and $F$, let
	\[ d^\infty_{\rm ham}(\Phi_H, \Phi_F) = \dbar_W (\Phi_H, \Phi_F) + \| \Hbar \# F \|^\infty_{\rm Hofer} = \dbar_W (\Phi_H, \Phi_F) + \| H - F \|^\infty_{\rm Hofer}, \]
and consider  the completion of the group of smooth Hamiltonian dynamical systems with respect to this stronger $L^\infty$-Hamiltonian metric.
Obviously, the groups of time-one maps satisfy $\Hameo^\infty (W,\omega) \subseteq \Hameo (W,\omega)$, since $d_{\rm ham}$ is controlled from above by $d_{\rm ham}^\infty$, but in fact the two groups are equal \cite{mueller:ghl08}.

Rather than with the contact metric $d_\alpha$, one may work with the stronger metric
	\[ d_\alpha^\infty(\Phi_H, \Phi_F) = \dbar_M (\Phi_H, \Phi_F) + | h - f | + \| H - F \|_\alpha^\infty, \]
where $\Phi_H$ and $\Phi_F$ are smooth contact isotopies, and 
\begin{align}\label{eqn:contact-length-inf}
	\| H \|_\alpha^\infty = \max_{0 \le t \le 1}  \left( \max_{x \in M} H (t,x) - \min_{x \in M} H (t,x) + \left| c_\alpha (H_t) \right| \right).
\end{align}
We call a triple $(\Phi, H, h)$ a \emph{continuous contact dynamical system} if it is the limit with respect to the metric $d_\alpha^\infty$ of a sequence $(\Phi_{H_i}, H_i, h_i)$ of smooth contact dynamical systems.
Restricting the contact metric $d_\alpha^\infty$ to the group of smooth strictly contact dynamical systems of $(M,\alpha)$ similarly defines \emph{continuous strictly contact dynamical systems}.
Strictly contact isotopies and their time-one maps were already studied in \cite{banyaga:ugh11}.
All of the results obtained in this paper for the group of topological contact dynamical systems also apply to continuous contact dynamical systems, with in many cases simpler proofs, since a limit contact Hamiltonian $H$ is now a continuous time-dependent function on $M$.

After adapting and streamlining the reparametrization techniques developed in \cite{mueller:ghl08, mueller:ghh08}, we obtain the main lemma of part II.
Every topological contact dynamical system is arbitrarily $d_\alpha$-close to a continuous contact dynamical system with the same time-one map, and in fact, the latter is smooth everywhere except possibly at time one.
In particular, $\Homeo^\infty (M,\xi) = \Homeo (M,\xi)$ and $\Homeo^\infty (M,\alpha) = \Homeo (M,\alpha)$ \cite{ms:tcd2}.
The second identity was obtained in \cite{banyaga:ugh11} in the special case that $\alpha$ is a regular contact form.

We also extend both the contact energy-capacity inequality and the Banyaga--Donato metric to the group of contact homeomorphisms and strictly contact homeomorphisms, respectively \cite{ms:tcd2}.

\subsection{On the contact Hamiltonian of a topological contact dynamical system}
For a topological Hamiltonian dynamical system of a symplectic manifold, the converse to the uniqueness of the isotopy was proved by Viterbo and Buhovsky--Seyfaddini, that is, if $(\Phi, H)$ and $(\Phi, F)$ are two topological Hamiltonian dynamical systems with equal topological contact isotopies, then the topological Hamiltonians $H$ and $F$ coincide.

\begin{lem} \label{lem:ham-equiv}
Let $(\Phi_i, H_i, h_i) \in \C (M,\alpha)$ be a sequence of smooth contact dynamical systems that converges with respect to the contact metric $d_\alpha$ to the topological contact dynamical system $(\Phi, H, h) \in \TC (M,\alpha)$.
Then the following are equivalent.
\begin{enumerate}
	\item [(i)] Suppose $(\Psi_i, F_i, f_i)$ is another sequence of smooth contact dynamical systems that converges with respect to the contact metric $d_\alpha$ to the topological contact dynamical system $(\Psi, F, f) \in \TC (M,\alpha)$.
	If $\Phi = \Psi$, then $H = F$, and $h = f$.
	\item [(ii)] If $\Phi$ is smooth, then $H$ and $h$ are smooth functions, $\Phi = \Phi_H$ is the smooth contact isotopy generated by the smooth contact Hamiltonian $H$, and $(\phi^t_H)^*\alpha = e^{h (t,\cdot)}\alpha$. 
	\item [(iii)] If $\Phi = \id$, then $H = 0$, and $h = 0$.
\end{enumerate}
\end{lem}

Without statement (ii), the analogous statement for Hamiltonian isotopies is well-known and first appeared in \cite{mueller:ghh07}.
A version including a statement similar to (ii) holds as well.

\begin{proof}
We prove that (i) implies (ii).
By contact rigidity (Theorem~\ref{thm:contact-rigidity}), $\Phi = \Phi_G$ for a smooth contact Hamiltonian $G$, and $(\phi_G^t)^* \alpha = e^{g (t,\cdot)} \alpha$.
Consider the constant sequence of smooth contact dynamical systems $(\Psi_i, F_i, f_i) = (\Phi_G, G, g)$.
Statement (i) implies $H = G$ and $h = g$, and the conclusion of statement (ii) holds.
That (ii) implies (iii) is again obvious.
Finally, we prove (iii) implies (i).
Indeed, the smooth sequence $(\Phi_i^{-1} \circ \Psi_i, \Hbar_i \# F_i, \holine_i \#f_i)$ by assumption converges in the contact metric to the topological contact dynamical system $(\id, \Hbar \# F, \holine \# f)$, and the conclusion of statement (iii) yields $\Hbar \# F = 0$ and $\holine \# f = 0$.
By equations~(\ref{eqn:gp-str-ham}) and (\ref{eqn:gp-str-fctn}), $H = F$ and $h = f$.
\end{proof}

We have already seen in Corollary~\ref{cor:unique-topo-conformal-factor-iso} that if $(\id, H, h)$ is a topological contact dynamical system, then $h = 0$.
The scheme of proof in \cite{buhovsky:ugh11} can then be adapted to the contact case.
The details are carried out in the sequel \cite{ms:tcd3}.

\nocite{banyaga:egs12, eliashberg:egc09, viterbo:eug06} % errata manually added in bbl file

\bibliography{contact} % bbl file manually altered
\bibliographystyle{amsalpha}

\end{document}